\documentclass[12pt]{amsart}
\usepackage{color}
\usepackage{amssymb}
\usepackage{comment}
\usepackage{pdfpages}
\usepackage{graphicx}
\usepackage{dcolumn}

\RequirePackage[numbers]{natbib}
\RequirePackage[colorlinks=true, pdfstartview=FitV, linkcolor=blue,
  citecolor=blue, urlcolor=blue]{hyperref}
\RequirePackage{hypernat}
\usepackage{paralist}

\usepackage{graphicx}
\graphicspath{{./figures/}}
\usepackage{amsfonts}
\usepackage{amsmath}
\usepackage{amsthm}
\usepackage{amssymb}
\usepackage{amsbsy}
\usepackage{epsfig}
\usepackage{fullpage}
\usepackage{natbib, mathrsfs} 
\usepackage{verbatim}
\usepackage[latin1]{inputenc}
\usepackage{mhequ}
\usepackage{algorithm}
\usepackage{algorithmic}

\numberwithin{equation}{section}

\def \be{\begin{equs}}
\def \ee{\end{equs}}
\def \P{\mathbb{P}}
\def \E{\mathbb{E}}

\def\propN{\lambda^{(N)}}

\newtheorem{theorem}{Theorem}[section]
\newtheorem{lemma}[theorem]{Lemma}
\newtheorem{remark}[theorem]{Remark}

\newtheorem{prop}[theorem]{Proposition}
\newtheorem{assumptions}[theorem]{Assumptions}
\newtheorem{assumption}[theorem]{Assumption}

\newtheorem{defn}[theorem]{Definition}

\theoremstyle{plain}
\newtheorem{thm}{Theorem}
\newtheorem*{thm-non}{Theorem}

\newtheorem{example}{Example}

\begin{document}

\title[Continuous Nonreversible Chains]{Bounds on Lifting Continuous-State Markov Chains to Speed Up Mixing}

\author{Kavita Ramanan$^{\ddag}$} 
\thanks{$^{\ddag}$kavita$\_$ramanan@brown.edu, 
   Division of Applied Mathematics
    Brown University, 182 George Street, Providence RI 02906, USA}
\author{Aaron Smith$^{\sharp}$}
\thanks{$^{\sharp}$smith.aaron.matthew@gmail.com, 
   Department of Mathematics and Statistics
University of Ottawa, 585 King Edward Drive, Ottawa
ON K1N 7N5, Canada}

\date{\today}

\maketitle

\begin{abstract}
It is often possible to speed up the mixing of a Markov chain $\{ X_{t} \}_{t \in \mathbb{N}}$ on a state space $\Omega$ by \textit{lifting},  that is, running a more efficient Markov chain $\{ \widehat{X}_{t} \}_{t \in \mathbb{N}}$ on a larger state space $\hat{\Omega} \supset \Omega$ that projects to $\{ X_{t} \}_{t \in \mathbb{N}}$ in a certain sense. In \cite{CLP99}, Chen, Lov{\'a}sz and Pak prove that  for Markov chains on finite state spaces, the mixing time of any lift of a Markov chain is at least the square root of the mixing time of the original chain, up to a factor that depends on the stationary measure of $\{X_t\}_{t \in \mathbb{N}}$. Unfortunately, this extra factor makes the bound in \cite{CLP99} very loose for Markov chains on large state spaces and useless for Markov chains on continuous state spaces. In this paper, we develop an extension of the evolving set method that allows us to refine this extra factor and find bounds for Markov chains on continuous state spaces that are analogous to the bounds in \cite{CLP99}. These bounds also allow us to improve on the bounds in \cite{CLP99} for some chains on finite state spaces.
\end{abstract}

\section{Introduction} 

Markov chain Monte Carlo (MCMC) is a widely-used method for approximately sampling from complicated distributions on some state space $\Omega$ \cite{Diac08}. 
The standard approach involves running a suitable ergodic Markov chain on $\Omega$ whose stationary distribution coincides with the sampling distribution and then sampling from the Markov chain after sufficiently many iterations. 
 The efficiency of this method depends on how long it takes for the underlying Markov chain to mix, that is,  for its law to come sufficiently close to  its stationary distribution (say, in the total variation metric).  
 For this reason, there is a great deal of interest in finding ways to construct Markov chains that decrease this 
mixing time.  Many of these methods \cite{Neal04, DiMi12, GSS10} are closely related to creating a \textit{lift} of a Markov chain defined on  the original state space $\Omega$.  Roughly speaking, this entails the construction of a more efficient Markov chain on a larger state space $\hat{\Omega} \supset \Omega$ that projects to a chain on the original state space in a suitable sense 
 (see Definition \ref{def-lift} for the precise definition of \textit{lift} used in this paper). This leads to the following natural question: how much efficiency can be gained by lifting? \par 

One version of this question was originally answered in Theorem 3.1 of \cite{CLP99}. For suitable definitions of \textit{mixing time} and \textit{lift}, the authors showed that if $\{ X_{t} \}_{t \in \mathbb{N}}$ is a Markov chain on a finite state space with stationary distribution $\pi$ and mixing time $\tau$, while $\{\widehat{X}_{t}\}_{t \in \mathbb{N}}$ is a lift of $\{ X_{t} \}_{t \in \mathbb{N}}$ with mixing time $\widehat{\tau}$, then  
\be \label{IneqClpThm31}
\widehat{\tau} \geq \frac{1}{10 \sqrt{30}} \sqrt{\frac{-1}{\log(\min_{x} \pi(x))}} \, \sqrt{\tau}.
\ee
In many examples, lifts can be constructed to achieve near-equality in this bound.  
For example, in \cite{diaconis2000analysis}, the authors construct a \textit{lift} of the usual simple random walk on the cycle $\mathbb{Z}_{n}$, for which $\pi$ is the uniform distribution on a set of cardinality $n \in \mathbb{N}$, and so the lower bound in  \eqref{IneqClpThm31} takes the form $c \sqrt{\tau}/\sqrt{\log (n)}$, with $c = 1/(10\sqrt{30})$. 
Theorem 1 of \cite{diaconis2000analysis} shows that the lift has mixing time $\widehat{\tau} = \Theta(n)$, while the mixing time of the original chain is $\tau = \Theta(n^{2})$.
This example still leaves open the possibility that the term $\log( \min_{x} \pi(x))$ in equation \eqref{IneqClpThm31} is not necessary - that is, it might be possible to replace this term by some universal constant $0 < c < \infty$.  However, as shown in Section 3 of \cite{CLP99}, this is not possible.  Theorem 3.2 of \cite{CLP99} implies the existence of a sequence of Markov chains and lifts, with mixing times $\{\tau_{n}\}_{n \in \mathbb{N}}$ and  $\{ \widehat{\tau}_{n} \}_{n \in \mathbb{N}}$ respectively, that satisfy
\be 
\lim_{n \rightarrow \infty} \frac{\widehat{\tau}_{n}}{\sqrt{\tau_{n}}} = 0.
\ee 

Unfortunately, the dependence on $\min_{x} \pi(x)$ in inequality \eqref{IneqClpThm31} means the inequality degrades with the size of the state space. For this reason, there seems to be no plausible way to use inequality \eqref{IneqClpThm31} to obtain any nontrivial bound on the mixing time of a lift of a Markov chain on a continuous or countably infinite state space. Our main result, Theorem \ref{LiftImpBd2}, provides a partial solution to this problem. It does so by establishing an inequality that is quite similar to inequality \eqref{IneqClpThm31}, with the key difference that the term $\min_{x} \pi(x)$ is replaced by a different quantity $\pi_{\ast}$ that can remain bounded away from zero even as the size of the underlying state space goes to infinity, or when the state space is continuous. \par

We briefly outline the proof of inequality \eqref{IneqClpThm31} in \cite{CLP99} and explain how we replace the term $\min_{x} \pi(x)$. The argument in \cite{CLP99} proceeds by relating the mixing time $\tau$ of a chain $\{X_{t}\}_{t \in \mathbb{N}}$ to another measure of mixing, the \textit{conductance} $\Phi$ of $\{X_{t}\}_{t \in \mathbb{N}}$, and then relating this conductance to the mixing time $\widehat{\tau}$ of a lift $\{\widehat{X}_{t}\}_{t \in \mathbb{N}}$ of  $\{X_{t}\}_{t \in \mathbb{N}}$.   The standard approach to  bounding a mixing time in terms of a conductance is well known to introduce a term of the form $\min_{x} \pi(x)$ (see, {\it e.g.}, Theorem 17.10 of \cite{LPW09}).  To avoid this problem, we establish a new bound on  the mixing time of a finite-state chain in terms of its conductance, which is  closely based on the \textit{evolving set method} of Morris and Peres \cite{MoPe06}. This new bound, obtained by analyzing the evolving set bound for Markov chains given a `warm start,' allows us to replace the term $\min_{x} \pi(x)$ in Theorem 17.10 of \cite{LPW09} with the larger term $\pi_{\ast}$, which has the critical property that it can be strictly positive and uniformly bounded away from zero for arbitrarily large state spaces.  This bound immediately results in a new lower bound for the mixing time of lifts of Markov chains, in the spirit of inequality \eqref{IneqClpThm31},  but which can be significantly better  (see Appendix A for an application of our bound to discrete state spaces). 
We then prove our analogue to  inequality \eqref{IneqClpThm31} for continuous state spaces by approximating our Markov chain 
with a sequence of Markov chains on discrete state spaces.  
 
The layout of this paper is as follows. In Section \ref{SecDefMainResults} we provide preliminary definitions, including the definition of a lift of a  Markov chain on a continuous state space, and state simplified versions of our main results.   In Section \ref{SecEvolvingSet}, we present our extension of the standard evolving set method and use  it to obtained improved bounds on  mixing times of Markov chains on discrete state spaces.  In Section \ref{SecLiftingBound}, we state and prove 
a bound on mixing times under fairly general conditions.  This is used in Section \ref{sec-last} to prove 
our main result: an  analogue to inequality \eqref{IneqClpThm31} for Markov chains on continuous states spaces. In Appendix \ref{AppSimpleImp}, we translate our new bounds back to the discrete setting of \cite{CLP99} and give an example showing that our bound can also provide  an improvement in that setting. Appendix \ref{AppSimpleCoup} proves a short technical result related to the example given in Section \ref{SecDefMainResults}.

\section{Notation, Definitions and Main Results} \label{SecDefMainResults}
 
\subsection{Notation and Definitions}
We let $\mathbb{R}^{+}$ be the nonnegative reals.

We recall the `big-O' notation. For functions $f,g$ on $[0,\infty)$ we say that $f = O(g)$ if there exist constants $0 <C, X < \infty$ such that $x > X$ implies $f(x) \leq C g(x)$.  For functions $f,g$ on $[0,\infty)$, we also say that $f = \mathbf{\Omega}(g)$ if there exist $0 <C, X < \infty$ such that $x > X$ implies $f(x) \geq C g(x)$. In the same situation, we say that $f = o(g)$ if $\lim_{x \rightarrow \infty} \frac{f(x)}{g(x)} = 0$. Finally, in the same situation, we write $f = \Theta(g)$ to mean that both $f = O(g)$ and $g = O(f)$. 
When $f$ and $g$ can be viewed as functions of two or more variables, we will use a subscript to indicate which variable is being compared if the bound does not hold uniformly in all variables. For example, if $f(x,y) = x^{2} y^{20}$ and $g(x,y) = x^{4} y^{10}$, we could write $f = O_{x}(g)$, since 
for every $y$ there exists some constant $0 < C = C(y) < \infty$ so that $f(x,y) \leq C(y) \, g(x,y)$ for all $x$. However, we could not write $f = O_{y}(g)$ or $f = O(g)$. 

For any set $S$, the function $\textbf{1}_{S}(x)$ if $x \in S$ and 0 otherwise. When $s$ is a single point, we slightly abuse notation and define $\textbf{1}_{s} = \textbf{1}_{\{s\}}$.

Throughout our paper, we will generally consider Markov chains on state spaces that are either a compact subset of Euclidean space $\mathbb{R}^{d}$ or the torus $\mathbb{T}^{d}$. We will always consider $\mathbb{R}^{d}$ with metric given by
\be \label{EqDefTorusMetric}
| x - y| := \max_{1 \leq i \leq d} | x_{i} - y_{i} |
\ee 
and equipped with the usual Borel $\sigma$-algebra and Lebesgue measure. When we are referring to either a torus or Euclidean space, we will denote by $\mathcal{A}$ the usual Borel $\sigma$-algebra on that space. We consider our torus to be equal to $\{ x \in \mathbb{R}^{d} \, : \, 0 \leq x_{i} < 1 \}$ as a set, with topology induced by the metric 
\be 
| x- y|  :=\max_{1 \leq i \leq d} \min (|x_{i} - y_{i} |, 1 - | x_{i} - y_{i} | )
\ee 
and equipped with the usual Borel $\sigma$-algebra and Lebesgue measure. \par 

For random variables $X,Y$ we will denote by $\mathcal{L}(X), \mathcal{L}(Y)$ their distributions and write $X \stackrel{D}{=} Y$ for $\mathcal{L}(X) = \mathcal{L}(Y)$. For any distribution $\mu$, we write $X \sim \mu$ for $\mathcal{L}(X) = \mu$. For any finite measure $\mu$ and measurable set $S$ with $\mu (S) > 0$,  we define the renormalized restriction $\mu |_{S}$ of $\mu$ to $S$ by setting $\mu |_{S} (A) = \frac{\mu(S \cap A)}{\mu(S)}$ for all measurable sets $A$. Finally, if $S$ is a finite state space, we denote by $\mathrm{Unif}(S)$ the uniform measure on $S$; if $S$ is a  compact measurable subset of $\mathbb{R}^{d}$ or $\mathbb{T}^{d}$ with positive Lebesgue measure, we denote by $\mathrm{Unif}(S)$ the probability measure inherited from the renormalized restriction of Lebesgue measure to $S$. \par 

We recall the definition of the \textit{Prokhorov distance} between two probability measures $\mu, \nu$ on a measure space $(\Omega, \mathcal{A})$ with metric $\phi$. For a set $S \in \mathcal{A}$ and constant $\delta > 0$, define the $\delta$-thickening of $S$ by
\be \label{EqDefThickening}
S_{\delta} := \left\{ x \in \Omega \, : \, \inf_{y \in S} \phi(x, y ) \leq \delta \right\}.
\ee 
The Prokhorov distance between $\mu$, $\nu$ is given by
\be 
d_{\mathrm{Prok}}(\mu,\nu) := \inf \{ \epsilon > 0 \, : \,  \forall A \in \mathcal{A}, \, \, \mu(A) \leq \nu(A_{\epsilon}) + \epsilon \text{ and } \nu(A) \leq \mu(A_{\epsilon}) + \epsilon  \}.
\ee 
When $X$, $Y$ are random variables, we will sometimes abuse this notation slightly and write $d_{\mathrm{Prok}}(X,Y)$ for $d_{\mathrm{Prok}}(\mathcal{L}(X),\mathcal{L}(Y))$. We recall  the following standard bound on the Prokhorov distance between two random variables $X,Y$ on the same probability space (see, \textit{e.g.}, the comment on Strassen's theorem in the proof of Theorem 3 of \cite{gibbs2002choosing} for reference):
\be \label{IneqProkBound}
d_{\mathrm{Prok}}(\mathcal{L}(X),\mathcal{L}(Y)) \leq \inf \{ \epsilon > 0 \, : \, \P[ \phi(X,Y) > \epsilon] \leq \epsilon \}. 
\ee

For any measure $\mu$ on a measure space and measurable function $f$ on the same space, we use the notation $\int_{x} f(x) \mu(dx)$ for the integral of $f$ with respect to $\mu$. Recall that a \textit{kernel} associated with a measure space $(\Omega,\mathcal{A})$ is a function $K \, : \, \Omega \times \mathcal{A} \mapsto [0,1]$ with the properties that:
\begin{enumerate}
\item the map $x \mapsto K(x,A)$ is $\mathcal{A}$-measurable for every $A \in \mathcal{A}$, and  
\item the map $A \mapsto K(x,A)$ is a probability measure on $(\Omega,\mathcal{A})$ for every $x \in \Omega$.
\end{enumerate}
A kernel is called \textit{reversible} with respect to a measure $\mu$ if
\be 
\int_{A} K(x,B) \mu(dx) = \int_{B} K(x,A) \mu(dx)
\ee  
for all measurable sets $A,B$. A measure $\mu$ is called a \textit{stationary measure} with respect to $K$ if
\be 
\int_{\Omega} K(x,A) \mu(dx) = \mu(A)
\ee 
for all measurable sets $A$. Recall that, if $\mu$ is reversible with respect to $K$, it is also stationary with respect to $K$. The kernel $K$ associated with $(\Omega, {\mathcal A})$ is said to be the \textit{generating} or \textit{transition} kernel $K$ of 
a Markov chain $\{X_{t}\}_{t \geq 0}$ with state space $\Omega$ if 
\be 
\P[X_{t+1} \in A | X_{t} = x] = K(x,A)
\ee 
for all $t \geq 0$ and all measurable sets $A$. 
 For $0 \leq \delta \leq 1$, we say that $K$ has \textit{holding probability at least $\delta$} if $\inf_{x \in \Omega} K(x, \{x\}) \geq \delta$. We also call such kernels $\delta$-lazy. 

Throughout this paper, our goal is to describe Markov chain mixing in terms of the popular \textit{total variation} metric. Recall that for a given measure space $(\Omega, \mathcal{A})$ and two measures $\mu, \nu$ on the $\sigma$-algebra $\mathcal{A}$, the total variation distance between $\mu$ and $\nu$ is given by:

\be
\vert \vert \mu - \nu \vert \vert_{TV} = \sup_{A \in \mathcal{A}} \left( \mu(A) - \nu(A) \right).
\ee
The \textit{mixing profile} of a Markov chain $\{ X_{t} \}_{t \in \mathbb{N}}$ (or more accurately, its generating kernel $K$) on the measure space $(\Omega,\mathcal{A})$ with stationary distribution $\pi$ is given by:
\be
\tau(\epsilon) := \inf \Big\{ t \, : \, \sup_{X_{0} = x \in \Omega} \vert \vert \mathcal{L}(X_{t}) - \pi \vert \vert_{TV} < \epsilon \Big\}.
\ee
By convention, the \textit{mixing time} of a Markov chain refers to $\tau := \tau \left( \frac{1}{4} \right)$.

\begin{remark} \label{RemarkEquivalenceMixingTimes}
\emph{Our notion of mixing time is slightly different from the notion of mixing time used in \cite{CLP99} (see Section 2 of that paper for their definition). We point out here that the main conclusion of the paper \cite{CLP99} holds for our definition as well: indeed, Theorem 1.4, Lemma 2.1 and Lemma 2.2 of \cite{PeSo11} together imply that our notion of mixing time and that used in \cite{CLP99} differ at most by multiplication by a universal constant factor for all Markov chains on finite state spaces. In particular, this means that it makes sense to talk about `improvements' to their Theorem 3.1, even though we use a different definition for the mixing time. }
\end{remark}

The purpose of our paper is to relate the mixing time of a kernel $K$ that  is $\frac{1}{2}$-lazy\footnote{The constant $\frac{1}{2}$ here is arbitrary; replacing the constant $\frac{1}{2}$ with any constant $0 < c < 1$ would result in similar bounds, though perhaps with slightly different constants.} to the mixing times of a class of lifted kernels.  In the literature on Markov chains, it is common to restrict attention to lazy kernels 
in order to avoid problems of parity that are not of great interest. This does mean, however, that our results do not apply to all kernels.  We now introduce our definition of the lift of a Markov transition kernel  on a continuous state space. 
 
\begin{defn} [Lift of a Markov transition kernel]
\label{def-lift}
Let $K$ be the generating kernel of a Markov chain $\{ X_{t} \}_{t \in \mathbb{N}}$ with stationary distribution $\pi$ on a measure space $(\Omega, \mathcal{A})$. We say that a kernel $\widehat{K}$ associated with a Markov chain $\{ \widehat{X}_{t} \}_{t \in \mathbb{N}}$ with stationary distribution $\widehat{\pi}$ on a measure space $(\widehat{\Omega},\widehat{\mathcal{A}})$ is a \textit{lift} of the kernel $K$ if there exists some measurable mapping $f \, : \, \widehat{\Omega} \rightarrow \Omega$ such that, for every  $S \in \mathcal{A}$ with $\pi(S) > 0$,  if $\mathcal{L}(\widehat{X}_{0}) = \widehat{\pi} \vert_{f^{-1}(S)}$ and $\mathcal{L}(X_{0}) = \pi \vert_{S}$, 
then 
\be \label{EqLiftDef}
f(\widehat{X}_{1}) \stackrel{D}{=} X_{1}. 
\ee
With some abuse of notation, we will sometimes also say that a Markov chain $\{\widehat{X}_t\}$ on $\widehat{\Omega}$ 
 is a lift of a Markov chain $\{X_t\}$ 
to mean that the generating kernel of $\{\widehat{X}_t\}$ is 
a lift of the generating  kernel  of $\{X_t\}$. 
\end{defn}

\begin{remark}
\label{rem-lift}
{\em 
Note that by choosing $S = \Omega$ in the definition above, we see that 
if $\mathcal{L}(\widehat{X}_{0}) = \widehat{\pi}$ and  $\mathcal{L}(X_{0}) = \pi$, then  equality \eqref{EqLiftDef}, along 
 with the fact that ${\mathcal L} ( \widehat{X}_1 )= \widehat{\pi}$ and ${\mathcal L}( X_1) = \pi$, implies that 
\[  \pi (S) = \widehat{\pi} (f^{-1} (S)),   \quad S \in {\mathcal A}. 
\]
In particular,  this implies that for any $S$ with $\pi(S) > 0$, we also have $\widehat{\pi} (f^{-1}(S)) > 0$ and so 
$\widehat{\pi}_{\vert_{f^{-1}(S)}}$ is well defined.   
Moreover, the relationship \eqref{EqLiftDef} is exactly
\be 
\int_{S} K(x,A) \pi(dx) =\int_{f^{-1}(S)} K(x,f^{-1}(A)) \widehat{\pi}(dx), \quad A \in {\mathcal A}, 
\ee 
for all measurable sets $A$. }
\end{remark}

This definition is a slight generalization of the notion of lifting in \cite{CLP99}. In particular, the observations in Remark \ref{rem-lift} show that our definition of lifting coincides with that in \cite[Section 3]{CLP99} when the state space $\Omega$ is countable and $\mathcal{A}$ is the usual $\sigma$-algebra on a countable state space. However, our definition also applies to Markov chains on general state spaces, in contrast to the definition in \cite{CLP99}, which only makes sense when $\Omega$ is finite or countable.  Throughout the paper, we will use the `hat' notation to refer to a lifted chain;  for example, $\widehat{\tau}$ is the mixing time of a lifted chain $\{ \widehat{X}_{t} \}_{t \in \mathbb{N}}$. \par 

\subsection{Statement of Main Results}
We now state Theorem \ref{LiftImpBd2}, which is  a weak version of our main result, Theorem \ref{LiftImpBd3}. Theorem \ref{LiftImpBd3}, which is proved in Section \ref{SecLiftingBound}, is much stronger than Theorem \ref{LiftImpBd2}  below; however, the associated assumptions, Assumptions \ref{AssumptionsRegularityConditions}, are more complicated to state than the conditions (stated as Assumptions \ref{AsmpSimpleIntroThm}), which are required for Theorem 1.   Although Assumptions \ref{AsmpSimpleIntroThm} consist of a long list of conditions, as explained in Remark \ref{RemarkMainThmAssumptions} below, these are intuitive and fairly straightforward to check. In Example \ref{ExSimpleCircle}, we give a short verification that all of them hold for a certain Markov chain on the unit circle.

\begin{assumptions} \label{AsmpSimpleIntroThm}
Suppose $\Omega$ is a compact subset of $\mathbb{R}^d$ or $\mathbb{T}^d$,  ${\mathcal A}$ is the associated Borel $\sigma$-algebra, and $K$ is a reversible transition kernel on $(\Omega, {\mathcal A})$  with stationary distribution $\pi$.
Assume that the following five properties are satisfied: 
\begin{enumerate}
\item For every $x \in \Omega$, there exist a constant $\alpha_x \in [\frac{1}{2}, 1)$, a countable collection $\mathcal{U}_{x}$ of elements of $\mathcal{A}$ and a probability distribution  $\{ K_{x,U} \}_{U \in \mathcal{U}_{x}}$ on $\mathcal{U}_{x}$ such that 
\be \label{EqKernelRepresentation}
\begin{split}
K(x,dy) 
&= \alpha_{x} \delta_{x} + (1-\alpha_{x}) \sum_{U \in \mathcal{U}_{x}} K_{x, U} \pi \vert_{U} (dy),
\end{split}
\ee 
and such that the map $x \mapsto \alpha_{x}$ is a measurable function from $\Omega$ to $[\frac{1}{2},1)$.

 Also, suppose that  $\pi_{\ast} > 0$, where 
\be \label{PiStarDef}
\pi_{\ast} :=\inf \{ \pi(U) \, : \, \exists \, x \in \Omega \, \text{ s.t. } \,  U \in \mathcal{U}_{x} \text{ and } K_{x,U} > 0 \}. 
\ee

\item  There exist constants $\beta, \gamma \in (0,  1]$ such that for all measurable sets $S$ with $\pi(S) \leq \beta$, we have 
\be \label{IneqAssEpsDeltFirstApp}
\int_{S}  (1-\alpha_{x}) \sum_{U \in \mathcal{U}_{x}} K_{x,U} \pi_{\vert_{U}}(dx) + \sup_{x \in S} \alpha_{x} \leq \gamma.
\ee 
\item For every $\epsilon > 0$, there exists $\delta > 0$ such that $\vert x - y \vert \leq \delta$ implies 
\be \label{IneqAssKernelCont}
|| K(x, \cdot) - K(y, \cdot) ||_{TV} \leq \min(\alpha_{x}, \alpha_{y}) + \epsilon.
\ee 

\item $\pi$ has a differentiable density $\rho$ with respect to Lebesgue measure that satisfies, for some $C, \, C^{-1}, \, D < \infty$,
\be \label{IneqDensityBlah}
C^{-1} < \rho(x) &< C \\
 | \rho'(x)| &< D
\ee 
for all $x \in \Omega$.

\item There exists some fixed $q \in \mathbb{N}, \psi > 0$ such that for every $x \in \Omega$ and $U \in {\mathcal U}_x$,
$K_{x,U} > 0$ implies that $U$ is the union of at most $q$ hypercubes of side length at least $\psi$.
\end{enumerate}
\end{assumptions}

\begin{thm} [Bound on Improvement Due to Lifting] \label{LiftImpBd2} 
Let  $K$ be a kernel satisfying Assumptions \ref{AsmpSimpleIntroThm}, and let $\tau$ denote its mixing time.   Then, for any lift $\widehat{K}$ of $K$, with associated mixing time $\widehat{\tau}$, we have: 
\be
\widehat{\tau} \geq  \frac{\sqrt{\gamma}}{16 \sqrt{ 2 \log \left( \frac{4}{\pi_{\ast}} \right) \log \left(\frac{2}{\sqrt{\beta}} \right) }} \sqrt{\tau},
\ee
where $\pi_{\ast}$ is the constant defined in equation \eqref{PiStarDef}.
\end{thm}


\begin{remark}
\label{RemarkMainThmAssumptions}
\emph{We comment briefly on the various conditions in Assumptions \ref{AsmpSimpleIntroThm}. These conditions split roughly into two groups: conditions (1-3) are fundamental to our approach, while conditions (4-5) are simple regularity conditions that, as shown in Theorem \ref{LiftImpBd3}, can be substantially weakened.}
\begin{enumerate}

\item \emph{This regularity condition is the main requirement of our method. It allows us to replace the term $\min_{x} \pi(x)$ in inequality \eqref{IneqClpThm31} (which is always zero for chains on continuous state spaces) with the term $\pi_{\ast}$ (which, by the right choice of representation in \eqref{EqKernelRepresentation}, can be made non-zero, even on continuous state spaces). As mentioned in the Introduction, no nontrivial bound on the mixing time is possible for Markov chains on continuous state spaces if this requirement is removed entirely. }

\emph{ Moreover, we point out that, for $1/2$-lazy Markov transition kernels $K$ on discrete state spaces, we can always find a representation analogous to \eqref{EqKernelRepresentation}:}
\be
K(x,y) = \alpha_{x} \delta_{x} + (1-\alpha_{x}) \sum_{U \in \mathcal{U}_{x}} K_{x, U} \pi \vert_{U} (y)
\ee
\emph{for some collection $\{\alpha_{x}\}_{x \in \Omega}$ and $\{\mathcal{U}_x\}_{x \in \Omega}$.  In particular,  for each $x \in \Omega$, one can then simply choose $\alpha_{x} = K(x,x)$,  $K_{x, \{y \}} = K(x,y)$ for all $y \in \Omega$ and $K_{x,U} = 0$ for all sets $U$ with $| U | > 1$.}

\item \emph{This bound is used to obtain a one-step minorization condition, which is useful when applying our evolving set bound. It is straightforward to generalize the requirement to a $k$-step minorization condition for some $k > 1$.}
\item \emph{This continuity condition holds for most natural chains, although it can be slightly weakened.}

\item \emph{This is not necessary, but allows our main result to be stated more simply. See Theorem \ref{LiftImpBd3} below for a more general version of the result which removes this requirement.}

\item \emph{This assumption is vastly stronger than what is required; it is only used to ensure that the boundaries of sets $U$ that satisfy $K_{x,U} > 0$ are not enormous compared to their volume. Theorem \ref{LiftImpBd3} below replaces this with a much weaker but more complicated sufficient condition.}
\end{enumerate}
\end{remark}

We provide an application of Theorem \ref{LiftImpBd2} to a simple random walk on the torus $\mathbb{T}$ that is analogous to the simple random walk on the cycle studied in \cite{diaconis2000analysis} and \cite{CLP99}:

\begin{example} \label{ExSimpleCircle}
\emph{Let $\Omega =\mathbb{T}$, fix $0 < c < \frac{1}{4}$, and define the kernel $K(x,\cdot)$ to be the mixture $\frac{1}{2} \delta_{x} + \frac{1}{2} \mathrm{Unif}(\{ y \in \mathbb{T} \, : \, |x - y | \leq c \} )$; this has stationary measure $\pi = \mathrm{Unif}(\mathbb{T})$. We claim that $K$ satisfies the assumptions of Theorem \ref{LiftImpBd2}. Going through the assumptions in the same order that they are given in Assumptions \ref{AsmpSimpleIntroThm}, we have:}

\begin{enumerate}
\item \emph{$K$ has holding probability exactly $\frac{1}{2}$. We now write $K$ in the form of  equation \eqref{EqKernelRepresentation}. We let $\alpha_{x} = \frac{1}{2}$, we let $\mathcal{U}_{x}$ be the singleton containing exactly the set  $\{ y \in \mathbb{T} \, : \, | x - y | < c \}$, and we let $K_{x, \{ y \in \mathbb{T} \, : \, | x - y | < c \} } = 1$. For this representation, we clearly have $\pi_{\ast} = 2c > 0$.}

\item  \emph{For any $0 < \beta < \frac{c}{2}$ and measurable set $S$ with $\pi(S) \leq \beta$,} 
\be 
\int_{x \in S}  (1-\alpha_{x}) \sum_{U \in \mathcal{U}_{x}} K_{x,U} \frac{1}{\pi(U)} \pi(dx) + \sup_{x \in S} \alpha_{x} \leq \frac{\beta}{2c} + \frac{1}{2}.
\ee 
\emph{Thus, for any $0 < \beta < \frac{c}{2}$, inequality \eqref{IneqAssEpsDeltFirstApp} is satisfied with constant $\gamma = \frac{\beta}{2c} + \frac{1}{2}$. Thus, this assumption is satisfied with constants $\beta = \frac{c}{4}$, $\gamma = \frac{5}{8}$.} 

\item \emph{ This property is satisfied with $\delta(\epsilon) = 2 c \epsilon$ because} 
\be 
\| K(x, \cdot) - K(y, \cdot) \|_{TV} &\leq \frac{1}{2} \| \delta_{x} - \delta_{y} \|_{TV} \\
&+  \frac{1}{2} \| \mathrm{Unif}( \{ z \in \mathbb{T} \, : \, |x-z| \leq c \}) - \mathrm{Unif}( \{ z \in \mathbb{T} \, : \, |y-z| \leq c \}) \|_{TV} \\
&\leq  \frac{1}{2} + \frac{|x - y|}{2c}.
\ee  

\item \emph{$\pi$ has density exactly 1.}

\item \emph{For all $x \in [0,1]$, there is a single set $U$ in ${\mathcal U}_x$, which  is a line segment of length $2c$.}
\end{enumerate}

\emph{Thus, Theorem \ref{LiftImpBd2} applies to this Markov chain. Since this random walk has mixing time $\tau = \Theta \left( c^{-2} \right)$ (see appendix \ref{AppSimpleCoup}), our result implies that any lift must have mixing time at least $\widehat{\tau} = \mathbf{\Omega} \left( \frac{1}{ c \, \log \left( c^{-1} \right)} \right)$. This example is closely analogous to the random walk on the cycle analyzed in \cite{diaconis2000analysis, CLP99}, with $c^{-1}$ playing the role of the size of the state space $n$. As in the analogous discrete walk, we do not know if this bound is tight - that is, whether it is possible to find a sequence of lifts with $\widehat{\tau} = O \left( \frac{1}{ c \, \log \left( c^{-1} \right)} \right)$ as $c^{-1} \rightarrow \infty$.} 
\end{example} 

\begin{remark}
\emph{We mention without proof that Theorem \ref{LiftImpBd2} can easily be extended to discrete Markov chains, either by making small modifications to our proof or by noting that a Markov chain on a finite state space $\Omega$ can easily be represented as a Markov chain on the continuous state space $\sqcup_{\omega \in \Omega} [0,1]_{\omega}$ - that is, the disjoint union of $|\Omega|$ copies of the interval $[0,1]$.}
\end{remark}

Our main tool for the proof of Theorem \ref{LiftImpBd2} will be a conductance bound. Given a transition kernel $K$ and stationary distribution $\pi$ of a Markov chain on a finite state space $\Omega$, for $t \in \mathbb{N}$, define 
\be \label{def-Kt}
K^{0}(x,y) &= \textbf{1}_{x}(y), \\
K^{t+1}(x,y) &= \sum_{z \in \Omega} K(x,z) K^{t}(z,y).
\ee
In the same setting,  for $x,y \in \Omega$, define the ergodic flow 
\be 
Q(x,y) := \pi(x) K(x,y)
\ee  and for $A, B \subset \Omega$, define 
\be \label{EqDefQCombinedChain}
Q(A,B) := \sum_{x \in A, y \in B} Q(x,y), \ee where 
the right-hand side is to be interpreted as zero when either $A$ or $B$ is the empty set.  When $\{x\}$ is a singleton, we will write $Q(x,B)$ for $Q(\{x\},B)$. Similarly, when $\{ y \}$ is a singleton, we will write $Q(A,y)$ for $Q(A, \{y\})$. Then define the conductance of a set $S$ with $\pi(S) > 0$ to be  
\be
\Phi(S) := \frac{Q(S, S^{c})}{\pi(S)}.
\ee
Analogously, for Markov chains on a continuous state space $\Omega \subset \mathbb{R}^{n}$  with transition kernel $K$ and stationary distribution $\pi$, define 
\be 
\Phi(S) := \frac{\int_{S}  K(x,S^{c}) \pi(dx)}{\pi(S)}.
\ee In both cases, define the conductance of the chain to be 
\be \label{EqDefConductanceAsInf}
\Phi := \inf_{S \, : \, 0<\pi(S) \leq \frac{1}{2}} \Phi(S). 
\ee 

\begin{assumption} \label{AssumptionsEvoSet}
Let $K$ be a $1/2$-lazy reversible transition kernel on a finite state space $\Omega$. 
Let $\pi$ be  the associated stationary distribution and let  $Q$ be defined as in  \eqref{EqDefQCombinedChain}.  
Assume that there exist constants $0 < \beta, \gamma < 1$ such that for any set $S \subset \Omega$ with $\pi(S) \leq \beta$ and $y \in \Omega$ with $\pi(y) > 0$, we have 
\be \label{EqEpsDelt1}
\frac{Q(S,y)}{\pi(y)} \leq 1- \gamma.
\ee
\end{assumption}

\begin{remark}
{\em We mention that inequality \eqref{EqEpsDelt1} below is the discrete analogue of inequality \eqref{IneqAssEpsDeltFirstApp}, and the condition in inequality \eqref{IneqAssEpsDeltFirstApp} is imposed exactly so that our discretized chains will satisfy inequality \eqref{EqEpsDelt1}. To reinforce this connection, we have used the same constants $\beta, \gamma$ in both equations, despite the fact that the inequalities are not identical.}
\end{remark}

We recall from Remark \ref{RemarkMainThmAssumptions}(1) that since $\Omega$ is finite, and the kernel $K$ 
is $1/2$-lazy, $K$ admits the representation  

 \be \label{EqFinChainKRep}
 K(x,\cdot) = \alpha_{x} \delta_{x}(\cdot) + (1 - \alpha_{x}) \sum_{U \subset \Omega} K_{x,U} \pi|_{U}(\cdot),
 \ee  
 where for all $x \in \Omega$ the constants $\{K_{x,U} \}_{ U \subset \Omega}$ are a probability distribution on the set of subsets of $\Omega$  and where $\frac{1}{2} \leq \alpha_{x} < 1$.   We fix such a representation, and 
define $\pi_{\ast}$ as in equation \eqref{PiStarDef}.

In this setting, we are able to establish the following conductance bound, which may be of independent interest. 

\begin{thm} [Uniform Evolving Set Bound] \label{EvSetBd1} 
Suppose the transition kernel $K$ satisfies Assumption \ref{AssumptionsEvoSet} and has conductance $\Phi \leq \frac{1}{4}$. Let $\tau$ be the associated mixing time,  and let $\pi_{\ast}$ be as defined above.  Then we have: 
\be \label{EqEvoSetThmRes}
\frac{1}{4 \Phi} \leq \tau \leq \frac{32}{\gamma \Phi^{2} } \log \left( \frac{4}{\pi_{\ast}} \right) \log \left(\frac{\sqrt{\beta}}{2} \right).
\ee
\end{thm}

This bound is non-trivial when $\pi_{\ast} > 0$, and 
can be viewed as a variant of Theorem 1 of \cite{MoPe06}; however it crucially differs from that bound in that it replaces $\pi_{\ast}' := \min_{x \in \Omega} \pi(x)$ with the larger $\pi_{\ast}$,  at the cost of introducing new terms that depend on $\beta$ and $\gamma$.  
Although we do not focus on this, Theorem \ref{EvSetBd1} also allows us, in some cases,  to improve the bounds in \cite{CLP99} even on discrete spaces when $\pi_{\ast}$ is much larger than $\pi_{\ast}'$ - see Appendix \ref{AppSimpleImp} for some details. \par

\section{Improved Evolving Set Bound} \label{SecEvolvingSet}

In this section, we prove Theorem \ref{EvSetBd1}. 

\begin{proof}[Proof of Theorem \ref{EvSetBd1}]
We follow the basic structure and notation of the proof of Theorem 17.10 in \cite{LPW09}, with the new ingredients being an improved mixing bound for `warm starts' and a minorization argument. We first recall the \textit{evolving set process} $\{ S_{t} \}_{t \in \mathbb{N}}$. For a transition kernel $K$ on finite state space $\Omega$, the evolving set process associated with $K$ is a Markov chain on subsets of $\Omega$.  Let $U_t, t \in \mathbb{N}$, be i.i.d.\ $\mathrm{Unif}([0,1])$ random variables.  To generate $S_{t+1}$ from $S_{t}$, set $S_{t+1} = \{ y \in \Omega \, : \, \frac{Q(S_{t},y)}{\pi(y)} \geq U_{t+1} \}$. We need a generalization of Lemma 17.12 from \cite{LPW09}, relating the transition kernel $K$ of the Markov chain to this evolving set process for starting distributions that are not concentrated on a single point. 
For $S \subseteq \Omega$ with $\pi(S) > 0$, define 
\be
\label{def-kts}
K^{t}(S,y) := \frac{1}{\pi(S)} \sum_{x \in S} \pi(x) K^{t}(x,y), 
\ee
with $K^t$ defined as in \eqref{def-Kt}.

\begin{lemma} [Transition Probabilities] \label{LemTransProb} 
Denote by $\{ S_{t} \}_{t \in \mathbb{N}}$ the evolving set process associated with $K$, and for $S \subseteq \Omega$, 
write $\P_{S_{0} = S}[\cdot]$ for $\P[\cdot \vert S_{0} = S]$ and let $\E_{S_{0} =S}$ denote expectation under $\P_{S_{0} = S}$. Then, for $t \in \mathbb{N}$ and $y \in \Omega$,
\be \label{IneqEvosetTranProb}
K^{t}(S,y) = \frac{\pi(y)}{\pi(S)} \P_{S_{0} = S}[y \in S_{t}].
\ee
\end{lemma}
\begin{proof}
We first prove that equality \eqref{IneqEvosetTranProb} holds by induction on $t$.
Recalling that $K^{0}$ is the identity matrix, it is clear that the equation holds when $t=0$. For $t > 0$, using 
\eqref{def-kts}, \eqref{def-Kt},    \eqref{IneqEvosetTranProb} and \eqref{EqDefQCombinedChain}, 
we then have that
\begin{align*}
K^{t+1}[S,y] &= \sum_{z \in \Omega} K^{t}(S,z) K(z,y) \\
&= \sum_{z \in \Omega} \frac{\pi(z)}{\pi(S)} \P_{S_{0} = S} [z \in S_{t}] K(z,y) \\
&=\frac{1}{\pi(S)} \sum_{z \in \Omega} \E_{S_{0} = S}[\textbf{1}_{\{z \in S_{t}\}} \pi(z) K(z,y)] \\
&= \frac{1}{\pi(S)} \E_{S_{0} = S}[Q(S_{t},y)] \\
&= \frac{\pi(y)}{\pi(S)} \P_{S_{0} = S}[y \in S_{t+1}],
\end{align*}
where the last equality follows from the dynamics of the evolving set process.  
This concludes the proof of \eqref{IneqEvosetTranProb}. 
\end{proof}

Next, for $t \in \mathbb{N}$, define $M_{t}$ to be equal to $S_{t}$ if $\pi(S_{t}) \leq \frac{1}{2}$, and to be equal to $S_{t}^{c}$ otherwise. Lemma 17.16 of \cite{LPW09} states that, for all $t \in \mathbb{N}$,
\be \label{EqRatioDec}
\E \left[ \sqrt{\frac{\pi(M_{t+1})}{\pi(M_{t})}} \Bigg| S_{t} \right] \leq 1 - \frac{\Phi^{2}}{2},
\ee
where $\Phi$ is defined in \eqref{EqDefConductanceAsInf}.
By Markov's inequality, this implies that for all $\alpha > 0$, 
\be \label{IneqMarkIneqEvoSet}
\begin{split}
\P_{S_{0} = S}[\pi(M_{t}) \geq \alpha] &\leq  \frac{\E_{S_{0} = S}[\sqrt{\pi(M_t)}]}{\sqrt{\alpha}}\\
& \leq \frac{\E_{S_{0} = S}[\sqrt{\pi(M_0)}]}{\sqrt{\alpha}} \left(1- \frac{\Phi^2}{2}\right)^t \\
&\leq \left( 1 - \frac{\Phi^{2}}{2} \right)^{t} \sqrt{\frac{\min(\pi(S), 1- \pi(S))}{\alpha}}.
\end{split}
\ee
In order to bound the probability that $M_{T} = \emptyset$, we consider times $T$ satisfying $T > B(A+1)$ for some $A,B>0$ and then estimate the chance that $M_{t}$ becomes empty during each of the first $A+1$ blocks of length $B$. We then bound the probability that $M_{A+1} = \emptyset$ uniformly in possible values of $M_{0}$, and take independent bounds for each region.  In particular, let $\beta$ be the constant in Assumption \ref{AssumptionsEvoSet} and note that, by inequalities \eqref{IneqMarkIneqEvoSet} and \eqref{EqEpsDelt1}, and the dynamics of $M_{t}$ and $S_{t}$, we have 
\be
\P_{S_{0} = S}[M_{A+1} = \emptyset] &\geq \P_{S_{0} = S}[\pi(M_{A}) \leq \beta] \, \inf_{S' \, : \, \pi(S') \leq \beta} \P[M_{A+1} = \emptyset \vert M_{A} = S'] \\
&\geq \left(1 - \left( 1 - \frac{\Phi^{2}}{2} \right)^{A} \sqrt{\frac{\pi(S)}{\beta}} \right) \\
&\times (1 - \sup  \{ \frac{Q(S',y)}{\pi(y)} \, : \, y \in \Omega, \, \pi(y) > 0, \, S' \subset \Omega, \, \pi(S') \leq \beta \} )\\
&\geq \left(1 -  \frac{( 1 - \frac{\Phi^{2}}{2} )^{A}}{\sqrt{\beta}} \right)\gamma, \\
\ee
where the last line uses Assumption \ref{AssumptionsEvoSet} and the fact that $\pi(S) \leq 1$.
Fix $ 0 < C < \infty$. Looking at $B = \lceil \frac{2 C}{ \gamma} \rceil$ sequences of length $A+1 = \lceil \frac{ \log ( \frac{\sqrt{\beta}}{2}) }{\log(1-\frac{\Phi^{2}}{2})} \rceil + 1$, this implies
\be \label{MainNullnessBound}
\begin{split}
\P_{S_{0} = S}[M_{B(A+1)} = \emptyset] &\geq 1 - \left( \sup_{U \subset \Omega} \P_{S_{0} = U}[M_{A+1} \neq \emptyset] \right)^{B} \\
&\geq 1 - \left( 1 -  \gamma \left(1 - \frac{(1 - \frac{\Phi^{2}}{2})^{A}}{\sqrt{\beta}} \right) \right)^{B} \\
&\geq 1 - \left(1 -  \frac{\gamma}{2} \right)^{B} \\
&\geq 1 - e^{-C}. 
\end{split}
\ee

Next, recall from Lemma \ref{LemTransProb} that, for all $t > 0$, all $S \subset \Omega$ and all $y \in \Omega$:
\be \label{EqToBeContinuedEvoSet}
\begin{split}
\vert K^{t}(S,y) - \pi(y) \vert &= \left| \frac{\pi(y)}{\pi(S)} \P_{S_{0} = S}[y \in S_{t}] - \pi(y) \right| \\
&= \frac{\pi(y)}{\pi(S)} \left| \P_{S_{0} = S}[y \in S_{t}] - \pi(S) \right|. \\
\end{split}
\ee
Let $\tau = \min \{ t \geq 0 \, : \, M_{t} = \emptyset \}$. It follows immediately from inequality \eqref{MainNullnessBound} that $\P[\tau < \infty] = 1$. Since $\{ \pi(S_{t})\}_{t}$ is a bounded martingale by Lemma 17.13 of \cite{LPW09}, the optional stopping theorem implies $\pi(S) = \P_{S_{0} = S}[S_{\tau} = \Omega]$. Thus, we can continue the calculation in equation \eqref{EqToBeContinuedEvoSet} by writing

\be
\vert K^{t}(S,y) - \pi(y) \vert &=  \frac{\pi(y)}{\pi(S)} \vert \P_{S_{0} = S}[y \in S_{t}] - \P_{S_{0} = S}[S_{\tau} = \Omega] \vert \\
&\leq \frac{\pi(y)}{\pi(S)} \Big( \vert \P_{S_{0} = S}[y \in S_{t}, \tau > t] - \P_{S_{0} = S}[S_{\tau} = \Omega, \tau > t] \vert \\
& \qquad \qquad + \vert \P_{S_{0} = S}[y \in S_{t}, \tau \leq t] - \P_{S_{0} = S}[S_{\tau} = \Omega, \tau \leq t] \vert \Big)\\
&\leq \frac{\pi(y)}{\pi(S)} \P_{S_{0} = S}[\tau > t] + 0,
\ee
where the last equality uses the property that $\emptyset$ is an absorbing state for the evolving set Markov chain, and so 
\be 
\{ \tau \leq t,  y \in S_t \}  = \{\tau \leq t, S_{\tau} \neq \emptyset\} 
= \{\tau \leq t, S_{\tau} = \Omega\}.
\ee 
Combining this with inequality \eqref{MainNullnessBound} and the definition of $A$ given above, 
for all $0 < C < \infty$ we have 
\be \label{SepBound}
\left| K^{ \lceil \frac{2 C}{ \gamma} \rceil \left( A + 1 \right)}(S,y) - \pi(y) \right|  \leq \frac{e^{-C}\, \pi(y)}{\pi(S)}.
\ee

Finally, we must relate $\frac{\vert K^{t}(S,y) - \pi(y) \vert } {\pi(y)}$ to $\max_{x,y} \frac{\vert K^{t}(x,y) - \pi(y) \vert } {\pi(y)}$. Using the decomposition of $K$ in terms of the weights $\alpha_{x}$ and $K_{x,U}$ from equation \eqref{EqFinChainKRep}, the definition in equation \eqref{def-kts}, as well as the decomposition \eqref{def-Kt} for  $K^{t+1}$, we note that 
\begin{align*}
\vert K^{t+1}(x,y) - \pi(y) \vert  &= |  (1 - \alpha_{x}) \sum_{U \subset \Omega} K_{x,U} \, K^{t}(U,y) + \alpha_{x} K^{t}(x,y) - \pi(y) \, ((1 - \alpha_{x}) \sum_{U \subset \Omega}  K_{x,U} + \alpha_{x} ) | \\
&= | (1 - \alpha_{x}) \sum_{U \subset \Omega} (K_{x,U} K^{t}(U,y) - \pi(y)) + \alpha_{x} (K^{t}(x,y) - \pi(y)) | \\
&\leq (1 - \alpha_{x})\sum_{U \subset \Omega} K_{x,U} \vert K^{t}(U,y) - \pi(y) \vert + \alpha_{x}. 
\end{align*}
Together, the previous two inequalities, the bound  $\alpha_x \geq \frac{1}{2}$  and equality \eqref{PiStarDef} show that

\be
\frac{1}{2} \sum_{y \in \Omega} \left| K^{ \lceil \frac{2 C}{ \gamma} \rceil \left( A + 1 \right) + 1}(x,y) - \pi(y) \right| & \leq \alpha_x + \sum_{y \in \Omega} \frac{e^{-C} \pi(y)}{2} \sum_{U \subset \Omega} \frac{K_{x,U}}{\pi(U)}  \\
&\leq \alpha_{x} + \frac{1}{2} \frac{e^{-C}}{\pi_{\ast}}.
\ee
Thus, for any $s \in \mathbb{N}$ and $T > \lceil \frac{2 }{ \gamma} \log \left( \frac{4}{\pi_{\ast}} \right) \rceil \left( A + 1 \right) + 1 + s$, we have by the same calculation
\begin{align*}
\vert \vert \mathcal{L}(X_{T}) - \pi \vert \vert_{TV} &\leq \frac{1}{2} \max_{x} \sum_{y \in \Omega}  \vert K^{T}(x,y) - \pi(y) \vert \\
&\leq \frac{1}{2} \max_{x} \left( \sum_{y \in \Omega}  \vert \alpha_{x} (K^{T-1}(x,y) - \pi(y)) + (1 - \alpha_{x}) (\sum_{U \subset \Omega} K_{x,U} K^{T-1}(U,y) - \pi(y)) \vert \right) \\
&\leq \frac{1}{2} \max_{x} \left( \alpha_{x} \sum_{y \in \Omega} |K^{T-1}(x,y) - \pi(y)| + (1 -\alpha_{x}) \sum_{U \subset \Omega} K_{x,U} |K^{T-s-1}(U,y) - \pi(y) | \right) \\
&\leq \ldots  \\
&\leq \max_{x \in \Omega} \alpha_{x}^{s} + \frac{1}{2} \max_{x} \sum_{U \subset \Omega} K_{x,U} |K^{T-s-1}(U,y) - \pi(y) | \\
&\leq \frac{1}{8} + \max_{x \in \Omega} \alpha_{x}^{s} . 
\end{align*}
We note that $\Phi \leq 1 - \max_{x \in \Omega} K(x,x) \leq 1 - \max_{x \in \Omega} K_{x, \{x\}}$, and so $K_{x,\{x\}}^{s} \leq \frac{1}{8}$ for $s > \frac{ - \log(8)}{\log (1 - \Phi)}$. Thus, for $T > \lceil \frac{2 }{ \gamma} \log \left( \frac{4}{\pi_{\ast}} \right) \rceil \left(\lceil \frac{ \log ( \frac{\sqrt{\beta}}{2}) }{\log(1-\frac{\Phi^{2}}{2})} \rceil + 1 \right) + 1 + \lceil \frac{ - \log(8)}{\log (1 - \Phi)} \rceil$, we have

\be
\vert \vert \mathcal{L}(X_{T}) - \pi \vert \vert_{TV} \leq  \frac{1}{4}.
\ee
Taking a Taylor expansion of the  $\log(1 - \Phi)$ term and using the assumption $\Phi \leq \frac{1}{4}$ completes the proof of the upper bound in inequality \eqref{EqEvoSetThmRes}. 

The lower bound is an immediate consequence of Theorem 7.3 of \cite{LPW09}.
\end{proof}

\section{Best Lifting Improvement Bound} \label{SecLiftingBound}

In this section, we state and prove Theorem \ref{LiftImpBd3}, which establishes 
a bound related to Theorem \ref{LiftImpBd2}, but under weaker assuptions.   In Section \ref{sec-last}, we 
derive Theorem \ref{LiftImpBd2} from Theorem \ref{LiftImpBd3}.

\subsection{A Generalization of Theorem \ref{LiftImpBd2}} \label{SeubsecGenThmPf}

We state some notation that will be used throughout Section \ref{SecLiftingBound}. Let $K$ be the transition kernel of a Markov chain on a compact subset $\Omega \subset \prod_{i=1}^{d} [A_{i}, B_{i}] \subset \mathbb{R}^{d}$ or $\Omega \subset \mathbb{T}^{d}$. Let $\pi$ be the associated stationary distribution and assume that $\pi$ is absolutely continuous, with some density $\rho$. For all $N \in \mathbb{N}$, define
\be 
\Omega^{(N)} = \{ N^{-1}(x_{1}, x_{2},\ldots, x_{d})  \, : \, N A_{i} - 2 \leq x_{i} \leq N B_{i} + 2, \, \, x_{1}, \ldots,x_{d} \in \mathbb{Z} \},
\ee 
and for all $x = (x_{1},\ldots,x_{d}) \in \Omega^{(N)}$, define
\be \label{EqDefSomeArbitraryDiscretization}
\pi^{(N)}(\{x\}) = \int_{x_{1}- \frac{1}{2N}}^{x_{1} + \frac{1}{2N}} \ldots \int_{x_{d}- \frac{1}{2N}}^{x_{d} + \frac{1}{2N}} \rho(y_{1},\ldots,y_{d}) dy_{1} \ldots dy_{d}.
\ee 
By a standard small abuse of notation, for $x \in \Omega^{(N)}$ we define $\pi^{(N)}(x) \equiv \pi^{(N)}(\{x\})$. We make similar small abuses of notation for the other measures and kernels defined on discrete state spaces.

Since the convex hull of $\Omega^{(N)}$ contains $\Omega$ and $\pi$ has no atoms, $\sum_{x \in \Omega^{(N)}} \pi^{(N)}(x) = 1$ and so $\pi^{(N)}$ is in fact a probability measure on $\Omega^{(N)}$. 
Assume that $K$ satisfies condition 1 of Assumptions \ref{AsmpSimpleIntroThm}  and hence, can be written in the form given by equation \eqref{EqKernelRepresentation} with associated  weights $\alpha_{x}$, $K_{x,U}$, $U \in {\mathcal U}_x, x \in \Omega$. 
We now define an associated kernel $L^{(N)}$ on $\Omega^{(N)}$ by setting 
\be 
\label{LN}
L^{(N)}(x,y) = \alpha_{x} \delta_{x}(y) + (1 - \alpha_{x}) \sum_{U \in \mathcal{U}_{x}, \, U \ni y} K_{x,U} \frac{\pi^{(N)}(y)}{\pi^{(N)}(U)}
\ee 
for all $x,y \in \Omega^{(N)}$. We define $K^{(N)}$ to be the Metropolis-Hastings kernel associated with the ``proposal" kernel $L^{(N)}$ and target distribution $\pi^{(N)}$. For $ x \neq y$, this kernel is given by $K^{(N)}(x,y) = L^{(N)}(x,y) \propN(x,y)$, where
\be 
\label{lambdaN}
\propN(x,y) =  \min \left(1, \frac{\pi^{(N)}(y) L^{(N)}(y,x)}{\pi^{(N)}(x) L^{(N)}(x,y)} \right).
\ee 
Otherwise, $K^{(N)}(x,x) = 1 - \sum_{y \in \Omega^{(N)}} K^{(N)}(x,y)$. Analogously to the definitions of $Q(x,y)$ and $Q(A,B)$ in equation \eqref{EqDefQCombinedChain}, define 
\be 
Q^{(N)}(x,y) = \pi^{(N)}(x) K^{(N)}(x,y)
\ee for $x,y \in \Omega^{(N)}$ and 
\be \label{EqNotConvEqLast}
Q^{(N)}(A,B) = \sum_{x \in A, y \in B} Q^{(N)}(x,y) \ee for $A,B \subset \Omega^{(N)}$. Throughout the remainder of the paper, we also abuse notation slightly and define $\pi^{(N)}(S) = \pi^{(N)}(S \cap \Omega^{(N)})$. We extend this slight abuse of notation to all other quantities that depend on subsets of $\Omega$ with size strictly greater than 1. For example, we define $Q^{(N)}(A,B) = Q^{(N)}(A \cap \Omega^{(N)},B \cap \Omega^{(N)})$ when $A, B \in \mathcal{A}$ and $A$, $B$ are not singletons. We will \textit{not} abuse notation in this way for single elements.

For $S \subset \Omega^{(N)}$, define the covering map $\mathcal{P}^{(N)} \, : \, 2^{\Omega^{(N)}} \mapsto 2^{\mathbb{R}^{d}}$, which sends a subset $S \subset \Omega^{(N)}$ to unions of rectangles that cover $S$, as follows:
\be \label{EqDefProjectionLikeOperator}
\mathcal{P}^{(N)}(S) = \bigcup_{x \in S} \prod_{i=1}^{d} \left[x_{i} - \frac{1}{2N}, x_{i} + \frac{1}{2N}\right] \subset \mathbb{R}^{d}.
\ee 
As has been our convention, we slightly abuse notation and define $\mathcal{P}^{(N)}(x) = \mathcal{P}^{(N)}(\{x\})$ for single elements $x \in \Omega^{(N)}$.

We also require some additional general notation. We define the collection  $\mathcal{B}$ of subsets that occur in our representation of $K$ as follows:
\be
\mathcal{B} :=  \bigcup_{x \in \Omega} \left\{ U \in  {\mathcal U}_x:  K_{x,U} > 0 \right\}. 
\ee 
For $\epsilon > 0$, we define the collection $\mathcal{R}_{\epsilon}$ of large rectangular subsets
\be 
\mathcal{R}_{\epsilon} := \left\{ \prod_{i=1}^{d} [a_{i}, b_{i}]  \subset \mathbb{R}^{d} \, : \, \inf_{i} (b_{i} - a_{i}) \geq \epsilon \right\}, 
\ee
and for $\epsilon > 0$ and $m \in \mathbb{N} \cup \{ \infty \}$, we define the  collection $\mathcal{H}_{\epsilon,m}$ of sets that are generated by a small number of large rectangular subsets
\be 
\mathcal{H}_{\epsilon,m} &:= \left\{ \cup_{j=1}^{k} R_{j}, (\cup_{j=1}^{k} R_{j})^{c} \, \, : \, 1 \leq k \leq m, \, R_{j} \, \mathrm{or} \, R_{j}^{c} \in \mathcal{R}_{\epsilon} \right\}. 
\ee

We write some necessary assumptions:

\begin{assumptions} \label{AssumptionsRegularityConditions}
Let $K$ be the reversible transition kernel of a Markov chain on a compact subset $\Omega \subset \mathbb{R}^{d}$, and let $\pi$ be the associated stationary distribution. Assume that
\begin{enumerate}
\item Parts \textbf{1} through \textbf{3} of Assumptions \ref{AsmpSimpleIntroThm} hold, with associated quantities $\alpha_x$, ${\mathcal U}_x$, $K_{x, U}$, $x \in \Omega$, ${\mathcal U}_x \subset {\mathcal A}$, $U \in {\mathcal U}_x$. 
\item For all fixed $0 < \epsilon < 1$, $m \in \mathbb{N}$ and $i \in \{1,2,3,4\}$, 
\be \label{EqRegCond1}
\lim_{N \rightarrow \infty} \theta^{(i)}(\epsilon,m,N) = 0,
\ee 
where 
\be
\theta^{(1)}(\epsilon,m,N)&= \sup_{S \in \mathcal{H}_{\epsilon,m}} | \pi^{(N)}(S) - \pi(S) |,  \\
\theta^{(2)}(\epsilon,m,N)&= \sup_{S  \in \mathcal{B}} | \pi^{(N)}(S) - \pi(S) |,  \\
\theta^{(3)}(\epsilon,m,N)&=   \sup_{S  \in \mathcal{B}} d_{\mathrm{Prok}}(\pi^{(N)}|_{S}, \pi|_{S}), \\
\ee 
and $\theta^{(4)}$ is given by
\be 
\theta^{(4)}(\epsilon,m,N) = \sup_{S \subset \Omega^{(N)}} \frac{\tilde{\theta}^{(4)}(S)}{\pi(\mathcal{P}^{(N)}(S))}, 
\ee 
with  
\be 
\tilde{\theta}^{(4)}(S) := & \left|\sum_{x \in S, y \notin S} (1-\alpha_{x}) \pi^{(N)}(x) \pi^{(N)}(y) \sum_{U \in \mathcal{U}_{x}, U \ni  y } K_{x,U} \frac{1}{\pi(U)} \right. \\
& \quad \left. - \int_{\mathcal{P}^{(N)}(S)}  (1-\alpha_{x}) \sum_{U \in \mathcal{U}_{x}} K_{x,U} \frac{\pi(U \cap \mathcal{P}^{(N)}(S^{c}))}{\pi(U)} \rho(x) dx  \right|,
\ee 
recalling the notational convention stated immediately following equation \eqref{EqNotConvEqLast}.
\end{enumerate}
\end{assumptions}

Our generalization of Theorem \ref{LiftImpBd2} is: 

\begin{thm} [Bound on Improvement Due to Lifting] \label{LiftImpBd3} 
Let $K$ be a reversible transition kernel  that satisfies Assumptions \ref{AssumptionsRegularityConditions} with associated constants $\gamma$ and $\beta$, and has mixing time $\tau$.  Then, for any lift $\widehat{K}$ of $K$  the mixing time $\widehat{\tau}$ of $\widehat{K}$ 
satisfies 
\be
\widehat{\tau}  \geq  & \frac{\sqrt{\gamma}}{32 \sqrt{  \log \left( \frac{4}{\pi_{\ast}} \right) \log \left(\frac{\sqrt{\beta}}{2} \right) }} \sqrt{\tau}, 
\ee
where $\pi_*$ is the constant defined in \eqref{PiStarDef}. 
\end{thm}

The main steps in the proof of Theorem \ref{LiftImpBd3} are the approximation result in Proposition \ref{PropConvDiscChain} below, which relates the kernel $K$ of interest to the sequence of discretized versions $K^{(N)}$ defined above, and a repeated application of the improved evolving set bound in Theorem \ref{EvSetBd1} to the discretized versions of the chain.  The approximation 
result is likely to be of use in more general contexts. 
Denote by $\tau$ and $\tau^{(N)}$  the mixing times of the transition kernels $K$ and $K^{(N)}$, respectively. Similarly, let $\Phi$ and $\Phi^{(N)}$ be the conductances of the transition kernels $K$ and $K^{(N)}$, respectively.  Also, define 
\be \label{def-instar}
\pi^{(N)}_{\ast} = \inf_{S \in \mathcal{B}} \pi^{(N)}(S).
\ee 
Then we have:

\begin{prop} \label{PropConvDiscChain}
Under the conditions of Theorem \ref{LiftImpBd3},
\be \label{lim-taun}
\mathrm{liminf}_{N \rightarrow \infty} \tau^{(N)} \geq   \tau   
\ee 
and
\be \label{lim-phin}
\lim_{N \rightarrow \infty} \Phi^{(N)} = \Phi.
\ee
Recall the constants $\beta, \gamma$ in inequality \eqref{EqEpsDelt1}. There exist $\gamma^{(N)}, \beta^{(N)}$ such that $Q^{(N)}$ satisfies inequality \eqref{EqEpsDelt1} with constants $\gamma^{(N)}$ and $\beta^{(N)}$ in place of $\gamma$ and $\beta$, respectively, and such that 

\be \label{EqLimitingBetaGammeDisc}
\lim_{N \rightarrow \infty} \beta^{(N)} &= \beta, \\
\lim_{N \rightarrow \infty} \gamma^{(N)} &= \gamma. 
\ee 
\end{prop}
We defer the proof of this proposition to Section \ref{SecProofProp}, and instead, first show how it implies Theorem \ref{LiftImpBd3}. 

\begin{proof}[Proof of Theorem \ref{LiftImpBd3}]
We first recall two results. From Section 3 of \cite{CLP99}\footnote{This proof is given for Markov chains on discrete state spaces, but the same bound holds for continuous state Markov chains with minor changes in notation}, the conductance $\Phi$ of a kernel is related to the conductance $\widehat{\Phi}$ of its lift by 
\be 
\widehat{\Phi} \leq \Phi.
\ee 
By Theorem 7.3 of \cite{LPW09}\footnote{Again, the result is stated for chains on discrete state spaces but applies more generally to continuous state Markov chains.}, the conductance gives a lower bound on the mixing time of a chain:
\be 
\widehat{\tau} \geq \frac{1}{4 \widehat{\Phi}}, \qquad \tau \geq \frac{1}{4 \Phi}.
\ee 
Using these two bounds, relations \eqref{lim-phin} and \eqref{EqLimitingBetaGammeDisc}  from Proposition \ref{PropConvDiscChain}, Theorem \ref{EvSetBd1}, and the definition of $\pi^{(N)}_{\ast}$ in equation \eqref{def-instar}, we calculate 

\be \label{IneqMainResult}
\begin{split}
\widehat{\tau} \geq \frac{1}{4 \widehat{\Phi}} 
&\geq \frac{1}{4 \Phi} \\
&= \lim_{N \rightarrow \infty} \frac{1}{ 4 \Phi^{(N)}} \\
&\geq  \liminf_{N \rightarrow \infty} \frac{\sqrt{\gamma^{(N)}}}{16 \sqrt{2  \log \left( \frac{4}{\pi_{\ast}^{(N)}} \right) \log \left(\frac{\sqrt{\beta^{(N)}}}{2} \right) }} \sqrt{\tau^{(N)}} \\
&=  \frac{\sqrt{\gamma}}{16 \sqrt{ 2 \log \left( \frac{4}{\pi_{\ast}} \right) \log \left(\frac{\sqrt{\beta}}{2} \right) }} \sqrt{\tau},
\end{split}
\ee
where the final equality holds from the fact that equality \eqref{EqRegCond1} with $i = 2$  implies $\lim_{N \rightarrow \infty} \pi_{\ast}^{(N)} = \pi_{\ast}$, the definition $\pi_* = \inf_{S \in {\mathcal B}} \pi(S)$, and the general relation 
\be 
| \inf_{S \in {\mathcal B}} \mu(S) - \inf_{S' \in {\mathcal B}} \nu(S)| \leq \sup_{S \in {\mathcal B}} | \mu(S) - \nu(S) | 
\ee 
for any probability measures $\mu$ and $\nu$ on $\Omega$. 
\end{proof}

\subsection{Proof of Proposition \ref{PropConvDiscChain}} \label{SecProofProp}

We begin with some technical lemmas.  

\begin{lemma} \label{KernelApproxLemma} 
Let the notation and assumptions be as in the statement of Proposition \ref{PropConvDiscChain}. For every $\epsilon > 0$, there exists $M = M(\epsilon)$ such that $N > M$ implies 
\be 
\label{Prok}
d_{\mathrm{Prok}}( K^{(N)}(x,\cdot), K(x,\cdot) ) \leq \epsilon
\ee 
for all $x \in \Omega^{(N)}$. Furthermore, for all $\epsilon > 0$ there exists some $M = M(\epsilon) < \infty$ so that $N > M(\epsilon)$ implies
\be \label{IneqKernelApproxLemmaSimpIneq3}
\sup_{x,y \in \Omega^{(N)}} (1 - \propN(x,y)) \leq \epsilon.
\ee
\end{lemma}
\begin{proof} 
Fix $\epsilon > 0$. By relation \eqref{EqRegCond1} with $i=3$, it is possible to choose $M_{1} = M_{1}(\epsilon)$ such that $N > M_{1}(\epsilon)$ implies $\sup_{S \in \mathcal{B}} d_{\mathrm{Prok}}(\pi^{(N)}|_{S}, \pi|_{S}) < \frac{\epsilon}{4}$.   Fix $N > M_{1}$. 
For each $x \in \Omega$, we now describe a coupling of the distributions $K(x,\cdot)$  and $K^{(N)}(x,\cdot)$. 
Let $\alpha_x, {\mathcal U}_x, K_{x,U}, {\mathcal U}_x \subset {\mathcal A}, U \in {\mathcal U}_x,$ be the quantities in 
\eqref{EqKernelRepresentation}, and recall $\pi$ is the stationary distribution of the kernel $K$.  
Consider a  probability space $(\Omega, {\mathcal F}, \mathbb{P})$ that is sufficiently large to support the following: 
\begin{enumerate}
\item 
A Bernoulli random variable $R_{1}$ with $\P(R_{1} = 0)  = \alpha_x$ and $\P(R_{1} = 1) = 1 - \alpha_x$; 
\item 
A $\mathcal{U}_x$-valued random element $R_2$, independent of $R_1$, such that 
$\P(R_2 = U) = K_{x,U}$, $U \in \mathcal{U}_x$. 
\item 
A pair of random variables $(R_3, R_3^{(N)})$ that couple 
$\pi |_{R_{2}}$ and $\pi^{(N)}|_{R_{2}}$ (i.e., with  $R_3 \sim \pi |_{R_{2}}$ and $R_3^{(N)} \sim \pi^{(N)}|_{R_{2}}$) 
with the property that 
\be
\label{R3-coup}
 \P\left[ | R_{3}^{(N)} - R_{3}| > \frac{\epsilon}{4}\right] \leq \frac{\epsilon}{4}. 
\ee
 Such a coupling is possible by \eqref{EqRegCond1} for $i=3$ and our choice of $N > M_{1}$. 
\item 
Finally, let $R_{4}^{(N)}$ be a $0$-$1$-Bernoulli random variable with 
$\P(R_{4}^{(N)} = 1) = \propN(x, R_{3}^{(N)})$. 
\end{enumerate}
We now define  random variables $X$ and $X^{(N)}$ on $(\Omega, {\mathcal F}, \mathbb{P})$ as follows: 
\begin{enumerate}
\item If $R_{1} = 0$, set $X = x$; otherwise, set $X = R_{3}$. 
 \item 
 If $R_{1} R_{4}^{(N)} = 0$, set $X^{(N)} = x$; otherwise, set $X^{(N)} = R_{3}^{(N)}$.  
\end{enumerate}
 In view of \eqref{EqKernelRepresentation}, it is easy to verify that $X \sim K(x, \cdot)$; and likewise, 
 from the definition of $K^{(N)}$, it is easy to see that $X^{(N)} \sim K^{(N)}(x,\cdot)$. 
Thus, for $N > M_1$, $(X,X^{(N)})$ is a coupling of $K(x,\cdot)$  and $K^{(N)}(x,\cdot).$

Now, note that on the event $\{R_4^{(N)} = 1\}$, $|X- X^{(N)}| = |R_3 - R_3^{(N)}| 1_{\{ R_1 = 1\}}$. 
Thus, \eqref{R3-coup} implies 
\be \label{IneqKernelApproxLemmaSimpIneq0}
\P\left[ \left\{| X - X^{(N)} | > \frac{\epsilon}{4} \right\} \cap \left\{R_{4}^{(N)} = 1 \right\} \right] \leq \frac{\epsilon}{4}.
\ee
On the other hand, 
\be 
\P\left[ \left\{| X - X^{(N)} | > \frac{\epsilon}{4} \right\} \cap \left\{R_{4}^{(N)} = 0 \right\} \right]  
\leq \P\left[R_{4}^{(N)} = 0\right] \leq \sup_{y} (1 - \propN(x,y)). 
\ee
Thus, to establish both \eqref{Prok} and \eqref{IneqKernelApproxLemmaSimpIneq3}, it suffices to show that for each $\epsilon > 0$, there exists $M_2 = M_2(\epsilon) > M_1$ such that for all $N > M_2$, 
\be
\label{toshow}
 \sup_{x,y \in \Omega^{(N)}} (1 - \propN(x,y)) \leq \frac{\epsilon}{4}. 
\ee
Since $\propN(x,y) = 1$ when $x = y$, it suffices to consider $x, y \in \Omega^{(N)}$ for which $x \neq y$. 
Now, for $x, y \in \Omega^{(N)}, x \neq y,$ by the definition of $\lambda^{(N)}$ in \eqref{lambdaN} and \eqref{LN},  
\be
\propN(x, y)  &  =  \min \left(1, \frac{\pi^{(N)}(y) \sum_{U \in \mathcal{U}_{y}, U \ni x} K_{y,U} \frac{\pi^{(N)}(x)}{\pi^{(N)}(U)}}
{\pi^{(N)}(x) \sum_{U \in \mathcal{U}_{x}, U \ni y} K_{x,U} \frac{\pi^{(N)}(y)}{\pi^{(N)}(U)}} 
\frac{(1 - \alpha_{y})}{(1 - \alpha_{x})} \right) \\
& =   \min \left(1, \frac{\sum_{U \in \mathcal{U}_{y}, U \ni x} K_{y,U} \frac{1}{\pi^{(N)}(U)}}
{\sum_{U \in \mathcal{U}_{x}, U \ni y} K_{x,U} \frac{1}{\pi^{(N)}(U)}} 
\frac{(1 - \alpha_{y})}{(1 - \alpha_{x})} \right). \\
\ee
By equality \eqref{EqRegCond1} with $i = 2$, for all $0 < \epsilon < 1$ there exists $M_{2} = M_{2}(\epsilon)$ (wlog chosen to be greater than $M_1$)  such that $N > M_{2}$ implies  
\be \label{IneqKernelApproxLemmaSimpIneq2}
\sup_{U \in \mathcal{B}} | \pi^{(N)}(U) - \pi(U) | < \frac{\epsilon}{100} \pi_{\ast}.
\ee 
We observe that $\pi$ is the stationary measure of the reversible kernel $K$, and so 

\be 
\frac{  \sum_{U \in \mathcal{U}_{y}, U \ni x} K_{y,U} \frac{1}{\pi(U)} }{\sum_{U \in \mathcal{U}_{x}, U \ni y} K_{x,U} \frac{1}{\pi(U)} } \frac{1 - \alpha_{y}}{1 - \alpha_{x}} &= \frac{ \rho(y) \sum_{U \in \mathcal{U}_{y}, U \ni x} K_{y,U} \frac{\rho(x)}{\pi(U)} }{\rho(x) \sum_{U \in \mathcal{U}_{x}, U \ni y} K_{x,U} \frac{\rho(y)}{\pi(U)} } \frac{1 - \alpha_{y}}{1 - \alpha_{x}} \\
&= \frac{\rho(y) K(y,dx)}{\rho(x) K(x,dy)} \\
& = 1. 
\ee 
 Combining this observation with the expression for $\lambda^{(N)}(x,y)$  and \eqref{IneqKernelApproxLemmaSimpIneq2}, we find that \eqref{toshow} holds 
for $N > M_{2}(\epsilon)$, which completes the proof of the lemma with $M (\epsilon) = M_2(\epsilon)$. 
\end{proof}

We then show: 

\begin{lemma} \label{LemmaUnifProkApprox}
Let the notation and assumptions of Theorem \ref{LiftImpBd3} hold. Fix $x \in \Omega$ and $T \in \mathbb{N}$. There exists a sequence of initial conditions $\{x^{(N)}\}_{N \in \mathbb{N}}$ such that, if  $X^{(N)}$ is the Markov chain with kernel $K^{(N)}$ started at $x^{(N)}$, and $X$ is the Markov chain with kernel $K$ started at $x$, then
\be \label{IneqPropMixTimeApprProk}
\limsup_{N \rightarrow \infty} \max_{0 \leq t \leq T} d_{\mathrm{Prok}}(X_{t}^{(N)}, X_{t}) = 0.
\ee 
\end{lemma}
\begin{proof}
For each $N \in \mathbb{N}$ and fixed $T \in \mathbb{N}$, we will construct a coupling  $\{X_{t}, X_{t}^{(N)}\}_{t =0}^{T}$ of the trajectories of the two Markov chains up to $T$ steps, with the property that for every $\epsilon  > 0$, 
\be \label{IneqPropMixTimeApprProkStrong}
\limsup_{N \rightarrow \infty} \P\left[ \max_{0 \leq t \leq T} |X_{t}^{(N)} - X_{t}| > \epsilon\right] = 0.
\ee 
We prove equality \eqref{IneqPropMixTimeApprProkStrong} by induction on $T$. For $T = 0$, by the definition of $\Omega^{(N)}$ in Section \ref{SeubsecGenThmPf}, it is always possible to choose $X_{0}^{(N)} = x^{(N)} \in \Omega^{(N)}$ such that $| x^{(N)} - x | \leq \frac{2d}{N}$. Thus, equality \eqref{IneqPropMixTimeApprProkStrong} holds for $T=0$. Next, we assume the equality holds up to time $t$, and we show that it holds up to time $t+1$. 
By the induction hypothesis, for all $\epsilon > 0$ there exists $M_{1} = M_{1}(t,\epsilon)$ such that $N > M_{1}$ implies
\be \label{IneqProkLemBasicIneq1}
\P \left[ \max_{0 \leq s \leq t} | X_{s}^{(N)} - X_{s} | > \frac{\epsilon}{10}\right] \leq \frac{\epsilon}{10}.
\ee  
Combining inequality \eqref{IneqAssKernelCont} with Lemma \ref{KernelApproxLemma}, for all $\epsilon > 0$ there exists $M_{2} = M_{2}(\epsilon,T)$ and $0 < \delta = \delta(\epsilon,t) \leq \epsilon$ such that $N > M_{2}$ implies the existence of a coupling of the four random variables $X_{t}^{(N)}, X_{t+1}^{(N)}, X_{t}, X_{t+1}$ with the property
\be \label{IneqProkLemBasicIneq2}
\P \left[ | X_{t+1}^{(N)} - X_{t+1} | > \frac{\epsilon}{10} \, \big| \, | X_{t}^{(N)} - X_{t} | < \delta\right] \leq \frac{\epsilon}{10}.
\ee
Finally, fix some $\epsilon > 0$. For $\delta = \delta(\frac{\epsilon}{3}, t)$ and $N > \max ( M_{1}(t, \delta), M_{2}(t, \frac{\epsilon}{3}))$, we have by inequalities \eqref{IneqProkLemBasicIneq1} and \eqref{IneqProkLemBasicIneq2} that
\be 
\P\left[ |X_{t+1}^{(N)} - X_{t+1} | > \frac{\epsilon}{10}\right] \leq \frac{2\epsilon}{10}.
\ee 
Combining this with inequality \eqref{IneqProkLemBasicIneq1}, we have
\be \label{IneqNeedTempHeloName}
\P\left[ \max_{0 \leq s \leq t+1} | X_{s}^{(N)} - X_{s} | > \frac{\epsilon}{10}\right] \leq \frac{3\epsilon}{10}.
\ee 
Recalling that $X_{s}^{(N)}$, $X_{s}$ are couplings of $K(x,\cdot)$ and $K^{(N)}(x,\cdot)$, inequality \eqref{IneqNeedTempHeloName} completes the induction step and completes the proof of equality \eqref{IneqPropMixTimeApprProkStrong}. Combining this with inequality \eqref{IneqProkBound} completes the proof of equality \eqref{IneqPropMixTimeApprProk}.
\end{proof}

We use the following standard lemma on approximation of sets with elementary sets:

\begin{lemma} \label{LemmaRegularApprox}
Let $\mu$ be a regular probability measure on $\mathbb{R}^{d}$ with the usual Borel $\sigma$-algebra $\mathcal{A}$. Then for any $A \in \mathcal{A}$ and any $\delta >0$, there exists a finite collection of sets $\{B_{i}\}_{i=1}^{m}$ of the form 
\be 
B_{i} = \prod_{j=1}^{d} [a_{j}, b_{j}]
\ee 
such that the symmetric difference $A \Delta (\cup_{i=1}^{m} B_{i})$ satisfies
\be 
\mu(A \Delta (\cup_{i=1}^{m} B_{i})) < \delta \mu(A).
\ee 
\end{lemma}

\begin{proof}
Since $\mu$ is a regular measure, there exists a countable collection $\{ B_{i} \}_{i \in \mathbb{N}}$ of sets of the form $B_{i} = \prod_{j=1}^{d} [a_{j}, b_{j}]$ that satisfy
\be \label{EqRegMeasureStart}
\begin{array}{rl}
\displaystyle 
A &\subset \cup_{i \in \mathbb{N}} B_{i} \\
\mu( (\cup_{i \in \mathbb{N}} B_{i}) \backslash A) &< \frac{\delta \, \mu(A)}{200}.
\end{array}
\ee 
Since $\{ \mu(\cup_{i =1}^{m} B_{i}) \}_{m \in \mathbb{N}}$ is a sequence of nonnegative increasing numbers that converges to $\mu(\cup_{i \in \mathbb{N}} B_{i}) \in [\mu(A), \mu(A)(1 + \frac{\delta}{200})) $, there exists some $m \in \mathbb{N}$ such that 
\be 
\mu(\cup_{i=1}^{m} B_{i}) > (1-  \frac{\delta}{200} ) \mu(A).
\ee 
This value of $m$ satisfies the requirements of the lemma.
\end{proof}

\begin{lemma} \label{LemmaNiceApproxCondSet}
Assume that the conditions of Theorem \ref{LiftImpBd3} hold. For every $\epsilon > 0$, there exists some $\delta > 0$, $m \in \mathbb{N}$ and $S \in \mathcal{H}_{\delta,m}$ such that 
\be \label{IneqNicApproxCondSetConcludingIneq}
\Phi(S) < \Phi + \epsilon.
\ee 
\end{lemma}

\begin{proof}
Since $\Phi = \inf_{0 < \pi(U) < \frac{1}{2}} \Phi(U)$ is defined as an infimum, there exists a set $U \subset \Omega$ with $0 < \pi(U) < \frac{1}{2}$ and 
\be 
\Phi(U) < \Phi + \frac{\epsilon}{100}.
\ee 
By Lemma \ref{LemmaRegularApprox}, there exists some $n \in \mathbb{N}$ and collection of sets  $\{B_{i}\}_{i=1}^{n}$ of the form 
\be 
B_{i} = \prod_{j=1}^{d} [a_{j}, b_{j}]
\ee 
such that  $S \equiv \cup_{i=1}^{n} B_{i}$ satisfies 
\be \label{EqRegMeasureEnd}
\pi(S \Delta U) \leq \frac{\epsilon \, \pi(U)}{50}.
\ee 
We now verify that this set $S$ satisfies inequality \eqref{IneqNicApproxCondSetConcludingIneq} by calculating: 
\be 
\Phi(S) &= \frac{Q(S,S^{c})}{\pi(S)} \\
&= \frac{1}{\pi(S)} \left( Q(S \cap U, S^{c}) + Q(S \cap U^{c}, S^{c}) \right) \\
&=  \frac{1}{\pi(S)} \left( Q(S \cap U, S^{c} \cap U^{c}) + Q(S \cap U, S^{c} \cap U)  + Q(S \cap U^{c}, S^{c}) \right) \\
&\leq \frac{1}{\pi(U) - \pi(S \Delta U)} \left( Q(U,U^{c}) + 2 \pi(S \Delta U) \right) \\
&\leq \frac{1+ \frac{\epsilon}{25}}{\pi(U)}  (Q(U,U^{c}) + \frac{\epsilon \pi(U)}{25}) \\
&\leq \Phi(U)(1 + \frac{\epsilon}{25}) + (1 + \frac{\epsilon}{25}) \, \frac{\epsilon}{25} \\
&\leq \Phi(U) + \epsilon,
\ee 
where the last line uses the fact that $\Phi(U) \leq 1$. This completes the proof of the lemma.
\end{proof}

Recall the definition of a \textit{thickening} in equation \eqref{EqDefThickening}. Whenever this notation is used in the remainder of this paper, we use the metric defined in equation \eqref{EqDefTorusMetric}. Note that, with this metric, if $S$ is a union of $d$-dimensional rectangles, then $S_{\delta}$ is also a union of $d$-dimensional rectangles. We have:

\begin{lemma}
Let Assumptions \ref{AssumptionsRegularityConditions} hold. For every fixed $0 < \epsilon < 1$ and $m \in \mathbb{N}$, the measure $\pi$ satisfies
\be \label{EqRegCond2}
\lim_{\delta \rightarrow 0}  \sup_{S \in \mathcal{H}_{\epsilon,m}} | \pi(S_{\delta}) - \pi(S) | = 0.
\ee 

\end{lemma}
\begin{proof}
Fix $ 0 < \epsilon < 1$ and $m \in \mathbb{N}$. Also fix $0 < C < \infty$, define $\rho_{C}(x) = \min(\rho(x), C)$ to be a truncated version of the density, and define $M(C) = \{x \, : \, \rho(x) > C \}$ to be the points where $\rho_{C}(x) \neq \rho(x)$. For any $S \in \mathcal{H}_{\epsilon,m}$,
\be 
| \pi(S_{\delta}) - \pi(S) | &= \int_{ S_{\delta} \backslash S} \rho(x) dx \\
&= \int_{(S_{\delta} \backslash S) \cap M(C)} \rho(x) dx +  \int_{(S_{\delta} \backslash S) \cap M(C)^{c}} \rho_{C}(x) dx \\
&\leq \pi(M(C)) +  \int_{ (S_{\delta} \backslash S) } \rho_{C}(x) dx  \\
&\leq \pi(M(C)) + O(C \delta).
\ee 
Since this holds for any $S \in \mathcal{H}_{\epsilon,m}$, we have 
\be 
\lim_{\delta \rightarrow 0}  \sup_{S \in \mathcal{H}_{\epsilon,m}} | \pi(S_{\delta}) - \pi(S) | \leq \lim_{\delta \rightarrow 0} (\pi(M(C)) + O(C \delta)) = \pi(M(C)).
\ee 
Since this holds for every $0 < C <   \infty$, this implies 
\be 
\lim_{\delta \rightarrow 0}  \sup_{S \in \mathcal{H}_{\epsilon,m}} | \pi(S_{\delta}) - \pi(S) | \leq \inf_{0 < C < \infty} \pi(M(C)) = 0,
\ee
completing the proof.
\end{proof}

Finally, we prove Proposition \ref{PropConvDiscChain}:
\begin{proof} [Proof of Proposition \ref{PropConvDiscChain}]
We first show that
\be 
\mathrm{liminf}_{N \rightarrow \infty} \tau^{(N)} \geq \tau. 
\ee 
Fix $X_{0} = x \in \Omega$, $T \in \mathbb{N}$, $\epsilon$, $m \in \mathbb{N}$ and $S_{0} \in \mathcal{H}_{\epsilon,m}$. If $\P[X_{T} \in S_{0}] \geq \pi(S_{0})$, set $S = S_{0}$; if not, set $S = S_{0}^{c}$. Now, fix $\delta > 0$. By Lemma \ref{LemmaUnifProkApprox}, there exists a sequence of starting points $X_{0}^{(N)} = x^{(N)} \in \Omega^{(N)}$ such that
\be \label{EqLimStartSimpleZQ2}
\mathrm{liminf}_{N \rightarrow \infty} \P[X_{T}^{(N)} \in S_{\delta}] \geq \P[X_{T} \in S].
\ee 
We observe that, for $\delta < \delta_{0}(\epsilon)$ sufficiently small, either the interior or the closure of $S_{\delta}$ is an element of $\mathcal{H}_{\epsilon',m}$ for some $\epsilon' > 0$.  Applying \eqref{EqRegCond1} with $i=1$, this implies
\be \label{EqLimEndSimpleZQ2}
\lim_{N \rightarrow \infty} \pi^{(N)}(S_{\delta}) = \pi(S_{\delta}).
\ee 
Combining equalities \eqref{EqLimStartSimpleZQ2}, \eqref{EqLimEndSimpleZQ2} with the triangle inequality and the fact that $\P[X_{T} \in S] \geq \pi(S)$, 
\be \label{EqLimTvProofSimp1}
\mathrm{liminf}_{N \rightarrow \infty} | \P[X_{T}^{(N)} \in S_{\delta}] - \pi^{(N)}(S_{\delta}) | \geq | \P[X_{T} \in S] - \pi(S) | - | \pi(S_{\delta}) - \pi(S) |.
\ee 
Now, assume that the mixing time $\tau$ satisfies $\tau > T$. By the definition of the mixing time, there exists a measurable set $R' \subset \Omega$, $x \in \Omega$ and $\eta > 0$ such that the chain $\{ X_{t} \}_{t \in \mathbb{N}}$ started at $X_{0} = x$ satisfies
\be 
\P[X_{T} \in R'] - \pi(R')  > \frac{1}{4} + \eta.
\ee 
By Lemma \ref{LemmaRegularApprox}, this implies that there exists $\epsilon > 0$, $m \in \mathbb{N}$ and $R \in \mathcal{H}_{\epsilon,m}$ such that
\be \label{ineqlastnamedthing}
\P[X_{T} \in R] - \pi(R) > \frac{1}{4} + \eta.
\ee 
We note that inequality \eqref{EqLimTvProofSimp1} holds for a generic set $S \in \mathcal{H}_{\epsilon,m}$, and in particular it holds with $S$ replaced by the set $R$ in inequality \eqref{ineqlastnamedthing}. Applying inequality \eqref{EqLimTvProofSimp1} to inequality \eqref{ineqlastnamedthing} implies that for some sequence of points $\{x^{(N)} \in \Omega^{(N)} \}_{N \in \mathbb{N}}$ and Markov chains started at $X_{0}^{(N)} = x^{(N)}$,
\be 
\mathrm{liminf}_{N \rightarrow \infty} | \P[X_{T}^{(N)} \in R_{\delta}] - \pi^{(N)}(R_{\delta}) | > \frac{1}{4} + \eta - | \pi(R) - \pi(R_{\delta}) |.
\ee 
Since this holds for all $\delta > 0$, equation \eqref{EqRegCond2} implies that 

\be 
\mathrm{liminf}_{N \rightarrow \infty} \| \mathcal{L}(X_{T}^{(N)}) - \pi^{(N)} \|_{\mathrm{TV}} \geq \frac{1}{4} + \eta > \frac{1}{4} .
\ee  
Since this argument holds for all $T < \tau$, we have
\be 
\mathrm{liminf}_{N \rightarrow \infty} \tau^{(N)} \geq \tau.
\ee 
\par 

Next, we show that
\be 
\mathrm{limsup}_{N \rightarrow \infty} \Phi^{(N)} \leq \Phi.
\ee
 To do so, fix $\epsilon > 0$. By Lemma \ref{LemmaNiceApproxCondSet}, there exists $\delta > 0$, $m \in \mathbb{N}$ and $S \in \mathcal{H}_{\epsilon,m}$ such that $\Phi(S) < \Phi + \epsilon$.   Applying \eqref{EqRegCond1} with $i=4$ and inequality \eqref{IneqKernelApproxLemmaSimpIneq3},
\be \label{EqCondEventual}
\begin{split}
\Phi^{(N)}(S) &= \frac{\sum_{x \in S, y \notin S} \pi^{(N)}(x) K^{(N)}(x,y)}{\pi^{(N)}(S)} \\ 
&= \frac{\sum_{x \in S, y \notin S} (1-\alpha_{x}) \pi^{(N)}(x) \pi^{(N)}(y) \sum_{U \in \mathcal{U}_{x}, y \in U} K_{x,U} \frac{1}{\pi(U)} } {\pi(S)} + o_{N}(1),
\end{split}
\ee 
where the error term $o_{N}(1)$ in this expression depends only on $N$ and not on the specific set $S$. 
We denote by $x^{\ast} \, : \, \mathbb{R}^{d} \mapsto \Omega^{(N)}$ a measurable function  that satisfies the containment condition
\be 
x \in \prod_{i=1}^{d} [x^{\ast}(x)-\frac{1}{2N},x^{\ast}(x) + \frac{1}{2N}]
\ee 
for all $x \in \mathbb{R}^{d}$.\footnote{We point out that such a function $x^{\ast}$ always exists. Note that the containment condition actually determines the value of $x^{\ast}$ except on the set $A = \{x \in \mathbb{R}^{d} \, : \, \exists \, 1 \leq i \leq d \, \text{ s.t. } x_{i} \in \frac{1}{N} \mathbb{Z} \}$. Furthermore, $x^{\ast}$ is continuous on $A^{c}$ and $A$ has measure 0. It is clear that the domain of $x^{\ast}$ can be extended from $A^{c}$ to $\mathbb{R}^{d}$ in a measurable way, e.g. by choosing the smallest allowed value in the usual lexicographic order on $\mathbb{R}^{d}$.}

From the definition of $\pi^{(N)}$ in equation \eqref{EqDefSomeArbitraryDiscretization}, the sum in the numerator in the second line of equation \eqref{EqCondEventual} can be written in the form 
\be 
\sum_{x \in S, y \notin S} (1-\alpha_{x}) \pi^{(N)}(x) \pi^{(N)}(y) \sum_{U \in \mathcal{U}_{x}, x^{\ast}(y) \in U} K_{x,U} \frac{1}{\pi(U)} &= \int F^{(N)}(x) \rho(x)  dx,
\ee 
where  
\be 
F^{(N)}(x) = (1 - \alpha_{x^{\ast}(x)}) \sum_{y \in \Omega^{(N)}} \pi^{(N)}(y) \sum_{U \in \mathcal{U}_{x^{\ast}(x)}, y \in U} K_{x^{\ast}(x),U} \frac{1}{\pi(U)} \leq 1.
\ee 

Since $F^{(N)} \leq 1$, we can apply the dominated convergence theorem along with equation \eqref{EqCondEventual} and inequality \eqref{IneqAssKernelCont} to conclude that
\be 
\lim_{N \rightarrow \infty} \frac{\sum_{x \in S, y \notin S} (1-\alpha_{x}) \pi^{(N)}(x) \pi^{(N)}(y) \sum_{U \in \mathcal{U}_{x}, \, U \ni y} K_{x,U} \frac{1}{\pi(U)} } {\pi(S)} = \Phi(S).
\ee 
Thus, we find that 
\be 
\mathrm{limsup}_{N \rightarrow \infty} \Phi^{(N)} &\leq \mathrm{limsup}_{N \rightarrow \infty} \Phi^{(N)}(S) \\
&= \Phi(S) \\
&\leq \Phi + \epsilon.
\ee 
Since this holds for all $\epsilon > 0$, we conclude that
\be 
\mathrm{limsup}_{N \rightarrow \infty} \Phi^{(N)}  \leq \Phi.
\ee 
\par 
Next, we show that 
\be 
\mathrm{liminf}_{N \rightarrow \infty} \Phi^{(N)} \geq \Phi.
\ee 
Fix $\epsilon > 0$ and consider a set $S^{(N)} \subset \Omega^{(N)}$. Recalling the definition of $\mathcal{P}^{(N)}$ from equation \eqref{EqDefProjectionLikeOperator}, let $S = \mathcal{P}^{(N)}(S^{(N)})$. By equality \eqref{EqCondEventual}, 
\be \label{EqCondEventual2}
\Phi^{(N)}(S^{(N)}) = \frac{\sum_{x \in S^{(N)}, y \notin S^{(N)}} (1-\alpha_{x}) \pi^{(N)}(x) \pi^{(N)}(y) \sum_{U \in \mathcal{U}_{x}, \, U \ni y} K_{x,U} \frac{1}{\pi(U)} } {\pi(S)} + o_{N}(1),
\ee 
where the $o_{N}(1)$ term does not depend on the particular set $S^{(N)}$. By equality \eqref{EqRegCond1} with $i = 4$, there exists $M_{1} = M_{1}(\epsilon)$ such that $N > M_{1}(\epsilon)$ implies 
\be 
|\sum_{x \in S^{(N)}, y \notin S^{(N)}} (1-\alpha_{x}) \pi^{(N)}(x) \pi^{(N)}(y) \sum_{U \in \mathcal{U}_{x}} K_{x,U} \frac{1}{\pi(U)} - \int_{S} \rho(x) (1-\alpha_{x}) \sum_{U \in \mathcal{U}_{x}} K_{x,U} \frac{\pi(U \cap S^{c})}{\pi(U)} dx | < \epsilon \pi(S).
\ee 
Combining this inequality with equality \eqref{EqCondEventual2}, there exists some $M_{2} = M_{2}(\epsilon)$ such that $N > M_{2}(\epsilon)$ implies 
\be 
\Phi^{(N)}(S^{(N)}) \geq \Phi(S) - \epsilon.
\ee 
Thus, for $N > M_{2}(\epsilon)$,
\be 
\Phi^{(N)} \geq \Phi - \epsilon,
\ee 
and so we conclude that
\be
\mathrm{liminf}_{N \rightarrow \infty} \Phi^{(N)} \geq \Phi. 
\ee

Finally, we prove the existence of sequences $\beta^{(N)}, \gamma^{(N)}$ that satisfy inequality \eqref{EqEpsDelt1}. We note that, if $\pi^{(N)}(S) \leq \beta$, then $\pi(\mathcal{P}^{(N)}(S)) = \pi^{(N)}(S) \leq \beta$ as well. Furthermore, by  equality \eqref{EqRegCond1} with $i=4$ and inequality \eqref{IneqKernelApproxLemmaSimpIneq3},  
\be 
\frac{Q^{(N)}(S,y)}{\pi^{(N)}(y)} &= \frac{\sum_{x \in S}  \pi^{(N)}(x) K^{(N)}(x,y)}{\pi^{(N)}(y)} \\
&= \sum_{x \in S} (1- \alpha_{x}) \pi^{(N)}(x) \sum_{U \in \mathcal{U}_{x}, \, U \ni y} K_{x,U} \frac{1}{\pi(U)} + \alpha_{y} \textbf{1}_{y \in S} + o_{N}(1).
\ee 
Combining this with the fact that $\pi(\mathcal{P}^{(N)}(S)) \leq \beta$, inequality \eqref{IneqAssKernelCont} and the assumption that inequality \eqref{IneqAssEpsDeltFirstApp} is satisfied, this implies
\be 
\frac{Q^{(N)}(S,y)}{\pi^{(N)}(y)} \leq \gamma +  o_{N}(1).
\ee
This proves the existence of a sequence $\{ \beta^{(N)}, \gamma^{(N)} \}_{N \in \mathbb{N}}$ with the desired properties, completing the proof of the proposition.
\end{proof}

\section{Derivation of Theorem \ref{LiftImpBd2} from Theorem \ref{LiftImpBd3}} 
\label{sec-last}

\begin{proof}[Proof of Theorem \ref{LiftImpBd2}]
It is sufficient to show that Assumptions \ref{AsmpSimpleIntroThm} imply that 
part (2) of Assumptions \ref{AssumptionsRegularityConditions} hold. We begin by proving equality \eqref{EqRegCond1} when $i=1$. Write $M = \max_{1 \leq i \leq d} (B_{i} - A_{i}) + 3$. Fix $\epsilon >0$, $m \in \mathbb{N}$ and $S \in \mathcal{H}_{\epsilon,m}$. For any $U \in \mathcal{R}_{\epsilon}$, there are at most $2d (4MN)^{d-1}$ boxes of the form $\prod_{i=1}^{d} [x_{i} - \frac{1}{2N}, x_{i} + \frac{1}{2N}]$ with $(x_{1},\ldots,x_{d}) \in \Omega^{(N)}$ that are partially but not entirely contained in $U$. Thus, taking a union bound over the (at most) $m$ sets in $S$, we have by condition  \eqref{IneqDensityBlah} that
\be 
| \pi^{(N)}(S) - \pi(S) | \leq Cm (2 M N)^{d-1} (2d) N^{-d}.
\ee 
Taking limits as $N \rightarrow \infty$, this implies 
\be 
\lim_{N \rightarrow \infty} \sup_{S \in \mathcal{H}_{\epsilon,m}} | \pi^{(N)}(S) - \pi(S) | \leq \lim_{N \rightarrow \infty} Cm (2 M N)^{d-1} \, (2d)  N^{-d} = 0.
\ee 
This completes the proof of equality \eqref{EqRegCond1} in the case $i=1$. By Assumptions \ref{AsmpSimpleIntroThm}, $\mathcal{B} \subset \mathcal{H}_{\psi, q}$, so this completes the proof of equality  \eqref{EqRegCond1} in the case $i=2$ as well.  

To prove equality \eqref{EqRegCond1} in the case $i=3$, we note that the first two parts of equality \eqref{EqRegCond1} imply that for all $M \in \mathbb{N}$,
\be 
\lim_{N \rightarrow \infty} \sup_{x \in \Omega^{(M)} } \sup_{S \in \mathcal{B}} \Big| \frac{  \pi^{(N)}(\mathcal{P}^{(M)}(x)) }{\pi^{(N)}(S)} -  \frac{  \pi(\mathcal{P}^{(M)}(x)) }{\pi(S)} \Big| = 0.
\ee 
Combining this with inequality \eqref{IneqProkBound}, this implies 
\be 
\lim_{N \rightarrow \infty} \sup_{S \in \mathcal{B}} d_{\mathrm{Prok}}(\pi^{(N)}|_{S}, \pi|_{S}) &\leq \lim_{N \rightarrow \infty} \sup_{S \in \mathcal{B}}  \inf \{ \epsilon > 0 \, : \, X \sim \pi^{(N)}|_{S}, \, Y \sim \pi|_{S}, \, \P[ \| X - Y \| > \epsilon] \leq \epsilon \} \\
&\leq \lim_{N \rightarrow \infty} (\frac{2 \sqrt{d}}{M} + \sup_{x \in \Omega^{(M)} } \sup_{S \in \mathcal{B}} \Big| \frac{  \pi^{(N)}(\mathcal{P}^{(M)}(x)) }{\pi^{(N)}(S)} -  \frac{  \pi(\mathcal{P}^{(M)}(x)) }{\pi(S)} \Big| ) \\
&= O \left( \frac{1}{M} \right).
\ee 
Since this holds for all $M \in \mathbb{N}$, this completes the proof of equality \eqref{EqRegCond1} in the case $i=3$. \par 
Finally, we recognize that the second term in the definition of $\tilde{\theta}^{(4)}(S)$ following equality \eqref{EqRegCond1} is exactly the integral $Q(\mathcal{P}^{(N)}(S),\mathcal{P}^{(N)}(S)^{c})$ while the first is a Riemann sum approximation of this integral. Thus, equality \eqref{EqRegCond1} in the case $i=4$ follows from the second part of inequality \eqref{IneqDensityBlah}, inequality \eqref{IneqAssKernelCont}, and Taylor's theorem. 
\end{proof}

\section*{Acknowledgements}
The first author, KR, is partially supported by NSF grant DMS-1407504 and the 
second author, AMS, is partially supported by an NSERC grant. This project was started while AMS was visiting ICERM, and he thanks ICERM for its generous support.

\bibliographystyle{plain}
\bibliography{CNRBBib}

\newpage

\appendix 

\section{Comparison to Previous Results on Discrete Spaces} \label{AppSimpleImp}
We point out that Theorem \ref{EvSetBd1} can give better bounds than Theorem 3.1 of \cite{CLP99} in some cases. Let $\{ X_{t} \}_{t \in \mathbb{N}}$ be a Markov chain on a finite state space that satisfies Assumptions \ref{AssumptionsEvoSet}, with associated conductance $\Phi$, mixing time $\tau$, and constants $\beta, \gamma > 0$ satisfying equation \eqref{EqEpsDelt1}. Then let $\{ \widehat{X}_{t} \}_{t \in \mathbb{N}}$ be a lift of $\{ X_{t} \}_{t \in \mathbb{N}}$ with mixing time $\widehat{\tau}$ and conductance $\widehat{\Phi}$. By the same calculation as contained in inequality \eqref{IneqMainResult} (but omitting the lines involving superscript $N$'s), we have:

\be \label{IneqDiscCor}
\widehat{\tau} \geq \frac{\sqrt{\gamma}}{32 \sqrt{  \log \left( \frac{4}{\pi_{\ast}} \right) \log \left(\frac{\sqrt{\beta}}{2} \right) }} \sqrt{\tau}
\ee

We give an example for which this bound is better than Theorem 3.1 of \cite{CLP99}. We fix $k, n \in \mathbb{N}$ and consider a random walk $\{ X_{t} \}_{t \in \mathbb{N}}$ on the cycle $\mathbb{Z}_{kn} = \{1,2,\ldots, nk \}$. Recall that the usual \textit{graph distance} on $\mathbb{Z}_{kn}$ is given by
\be
d(x,y) = \min(|x-y|, nk - |x-y|)
\ee
We then define the transition kernel $K$ by:

\be \label{EqKernelLiteralHeloAppendixA}
\begin{split}
K(i,i) &= \frac{1}{2} \\
K(i,j) &= \frac{1}{2k}, \qquad 0 < d(i,j) \leq k \\
K(i,j) &= 0, \qquad d(i,j) > k. \\
\end{split}
\ee
This walk has stationary distribution $\pi(x) = \frac{1}{kn}$ for all $x \in \mathbb{Z}_{kn}$. It can be represented according to the form of equation \eqref{EqKernelLiteralHeloAppendixA} using the weights 
\be \label{EqKernelRepAppendixA}
\begin{split}
K_{x,\{x\}} &= \frac{1}{2}, \\
K_{x, \{x-k,\ldots,x-1,x+1,\ldots,x+k\}} &= \frac{1}{2}, \\
K_{x,U} &= 0 \qquad \text{otherwise.}
\end{split}
\ee 
Using arguments identical to those in Theorem 2 of Chapter 3 of \cite{Diac88}, the mixing time $\tau$ of this walk can be shown to be $\tau = \Theta(n^{2})$. Theorem 3.1 of \cite{CLP99} (combined with remark \ref{RemarkEquivalenceMixingTimes}) implies that any lift of this chain must have mixing time at least $\widehat{\tau} = \Omega \left( \frac{n}{\sqrt{\log(kn)}} \right)$. 

In the other direction, set $\beta = \frac{1}{12n}$, $\gamma = \frac{3}{4}$. For any set $S \subset \mathbb{Z}_{kn}$ with $\pi(S) \leq \frac{1}{12n}$,
\be 
\frac{Q(S,y)}{\pi(y)} \leq \frac{1}{2} + \frac{1}{2k} | \{ x \in S \, : \, 1 \leq d(x,y) \leq k\} | \leq \frac{3}{4}. 
\ee 
Thus, part \textbf{(1)} of Assumptions \ref{AssumptionsEvoSet} is satisfied with $\beta = \frac{1}{12n}$ and $\gamma = \frac{3}{4}$. Using the representation of $K$ given by equation \eqref{EqKernelRepAppendixA}, we have $\pi_{\ast} = \frac{2}{n}$. By inequality \eqref{IneqDiscCor}, then, we find $\widehat{\tau} = \Omega \left( \frac{n}{\log(n)} \right)$.

Let $\{ k(\lambda) \}_{ \lambda \in \mathbb{N}}$ be a sequence with the property $k(\lambda) \gg n(\lambda)^{\log(n(\lambda))}$. Then let $\{X_{t}^{(\lambda)}\}_{t \geq 0}$ be the Markov chain with kernel given by equation \eqref{EqKernelLiteralHeloAppendixA} with $k = k(\lambda)$. In this regime, the bound from inequality \eqref{IneqDiscCor} is tighter than that of Theorem 3.1 of \cite{CLP99}. 
More generally, we expect inequality \eqref{IneqDiscCor} to be tighter than  Theorem 3.1 of \cite{CLP99} when the discrete state space $\Omega$ of the Markov chain $\{X_{t}\}_{t \geq 0}$ is extremely large compared to the support of $K(x,\cdot)$ for all $x$.

\section{ Mixing Bound for the Continuous Cycle Walk} \label{AppSimpleCoup}

In this section, we show that the kernel $K$ from example \ref{ExSimpleCircle} has mixing time $\tau = \Theta(c^{-2})$.  Let $\mu$ be the uniform measure on $[-c,c]$ as a subset of the torus $[0,1]$. It is well known that, as a compact abelian Lie group, the characters of the torus $[0,1]$ are given by $\chi_{n}(x) = e^{2 \pi inx}$. We calculate 
\be
\int_{[0,1]} \chi_{n}(x) d\mu(x) &= \frac{\sin(2 \pi n c)}{2 \pi n c}.
\ee
By Lemma 4.3 of \cite{Rose94}, this means that for $A >1$, $T > A c^{-2}$ and $S \subseteq [0,1]$ measurable, we have for all $c < C_{0}$ sufficiently small that
\be
 \vert K^{T}(x,S) - \pi(S) \vert &\leq \frac{1}{4} \sum_{n >1} \left(\frac{\sin(2 \pi n c)}{2 \pi n c}\right)^{T} \\
&\leq \frac{1}{4} \sum_{n=2}^{\lfloor \frac{1}{20 \pi \, c} \rfloor} \left(\frac{\sin(2 \pi n c)}{2 \pi n c}\right)^{T} + \frac{1}{4} \sum_{n=\lfloor \frac{1}{20 \pi \, c} \rfloor}^{\lceil \frac{1}{4 c} \rceil} \left(\frac{\sin(2 \pi n c)}{2 \pi n c}\right)^{T} + \frac{1}{4}\sum_{n > \lceil \frac{1}{4c} \rceil} \left(\frac{\sin(2 \pi n c)}{2 \pi n c}\right)^{T} \\
&\leq \frac{1}{2} \sum_{n=2}^{\lfloor \frac{1}{20 \pi \, c} \rfloor} \left(\frac{\sin(2 \pi n c)}{2 \pi n c}\right)^{T} + \frac{1}{2} \sum_{n=\lfloor \frac{1}{20 \pi \, c} \rfloor}^{\lceil \frac{1}{4 c} \rceil} \left(\frac{\sin(2 \pi n c)}{2 \pi n c}\right)^{T} + \frac{1}{4}\sum_{n > \lceil \frac{1}{2c} \rceil} \left(\frac{\sin(2 \pi n c)}{2 \pi n c}\right)^{T} \\
&\leq \frac{1}{2} \sum_{n=2}^{\lfloor \frac{1}{20 \pi \, c} \rfloor}  \left( 1 - \frac{9}{10} \frac{(2 \pi n c)^{2}}{6} \right) + \frac{1}{2} \sum_{n=\lfloor \frac{1}{20 \pi \, c} \rfloor}^{\lceil \frac{1}{4 c} \rceil} \left( 0.999 \right)^{T} + \frac{1}{4}\sum_{n > \lceil \frac{1}{2c} \rceil} \left(\frac{1}{2 \pi n c}\right)^{T}  \\
&\leq \frac{1}{2} \sum_{n=2}^{\infty} e^{- \frac{36 \pi^{2}}{60} A \, n^{2}} + c^{-1}  \left( 0.999 \right)^{A c^{-2}} +  \frac{1}{4} \int_{\frac{1}{2c} -2}^{\infty} \left(\frac{1}{2 \pi n c}\right)^{T} dx  \\
&\leq e^{- \frac{36 \pi^{2}}{60} A } +  c^{-1}  \left( 0.999 \right)^{A c^{-2}} + \frac{1}{4} \frac{1}{\frac{1}{2c} - 2} \frac{1}{A c^{-2} - 1} \pi^{-A c^{-2}}.
\ee 

Since each of the three terms in the final line can be made arbitrarily small (uniformly in $c < C_{0}$) by choosing $A$ large, this implies $\tau = O(c^{-2})$. \\

To show the reverse inequality, let $T(c) = A_{c} c^{-2}$, for some sequence $A_{c} \rightarrow 0$ as $c \rightarrow \infty$. For a copy of the chain started at $X_{0}=0$, we have by Bernstein's inequality $P[ \vert X_{T(c)} \vert > \frac{1}{10}] \leq 2 e^{- \frac{3}{2 A_{c}} \frac{1}{30 + c} } \rightarrow 0$. Thus, $\limsup_{c \rightarrow \infty} \vert \vert \mathcal{L}(X_{T(c)}) - U \vert \vert_{TV} \geq \frac{1}{2}$, and so $\tau = \Omega(c^{-2})$.

\end{document}